\documentclass[12pt,a4paper,psamsfonts]{amsart}

\usepackage{fancyhdr}
\usepackage{appendix}
\usepackage{amssymb,amscd,amsxtra,calc}
\usepackage{mathrsfs}
\usepackage{cmmib57}
\usepackage{multirow}
\usepackage[all]{xy}
\usepackage{longtable}
\usepackage[colorlinks,linkcolor=blue,anchorcolor=blue,citecolor=green,backref=page]{hyperref}
\usepackage{mathabx,epsfig}
\def\acts{\ \rotatebox[origin=c]{-90}{$\circlearrowright$}\ }
\def\racts{\ \rotatebox[origin=c]{90}{$\circlearrowleft$}\ }

\numberwithin{equation}{section} 

\theoremstyle{plain}
    \newtheorem{thm}{Theorem}[section]

\newtheorem{Conjecture}[thm]{Conjecture}
     \newtheorem{conj}[thm]{Conjecture}
    \newtheorem{cor}[thm]{Corollary}
\newtheorem{lem}[thm]{Lemma}
    \newtheorem{prop}[thm]{Proposition}
    \newtheorem{ques}[thm]{Question}
    \newtheorem{theorem}[thm]{Theorem}

\theoremstyle{definition}
        
     \newtheorem{defn}[thm]{Definition}
    \newtheorem*{notation*}{Notation and Terminology}
    \newtheorem{rem}[thm]{Remark}
    \newtheorem{cl}[thm]{Claim}
\theoremstyle{remark}

    \newtheorem{setup}[thm]{}

\DeclareMathOperator{\alb}{alb}
\DeclareMathOperator{\Alb}{Alb}

\DeclareMathOperator{\End}{End}
\DeclareMathOperator{\Exc}{Exc}

\DeclareMathOperator{\Hom}{Hom}
\DeclareMathOperator{\id}{id}

\DeclareMathOperator{\Ker}{Ker}

\DeclareMathOperator{\N}{N}
\DeclareMathOperator{\Nef}{Nef}

\DeclareMathOperator{\NS}{NS}

\DeclareMathOperator{\PE}{PE}
\DeclareMathOperator{\Per}{Per}

\DeclareMathOperator{\Pic}{Pic}
\DeclareMathOperator{\pr}{pr}

\DeclareMathOperator{\Prep}{Prep}

\DeclareMathOperator{\SEnd}{SEnd}
\DeclareMathOperator{\Sing}{Sing}
\DeclareMathOperator{\Spec}{Spec}
\DeclareMathOperator{\Supp}{Supp}

\DeclareMathOperator{\Tor}{Tor}

\DeclareMathOperator{\sHom}{\mathscr{H}\kern -.3pt \mathit{om}}

\newcommand{\alg}{\mathrm{alg}}
\newcommand{\reg}{\mathrm{reg}}

\newcommand{\C}{\mathbb{C}}

\newcommand{\PP}{\mathbb{P}}

\newcommand{\Q}{\mathbb{Q}}

\newcommand{\R}{\mathbb{R}}

\newcommand{\Z}{\mathbb{Z}}

\theoremstyle{definition}

\newtheorem*{ack}{Acknowledgments}

\newif\ifhascomments \hascommentstrue
\ifhascomments
  \newcommand{\changemade}[1]{{\color{blue}[[\ensuremath{\spadesuit\spadesuit\spadesuit} #1]]}}
   \newcommand{\todo}[1]{{\color{red}[[\ensuremath{\spadesuit\spadesuit\spadesuit} #1]]}}
\else
  \newcommand{\todo}[1]{}
  \newcommand{\changemade}[1]{}
\fi

\makeatletter
\@namedef{subjclassname@2020}{\textup{2020} Mathematics Subject Classification}
\makeatother

\title[Non-density of small points]
{Non-density of points of small arithmetic degrees}
\author[Y.~Matsuzawa, S.~Meng, T.~Shibata, D.-Q.~Zhang]
{Yohsuke Matsuzawa, Sheng Meng, Takahiro Shibata, and De-Qi Zhang}
\date{}

\address{
 \textsc{Department of mathematics} \endgraf
\textsc{Osaka metropolitan university,
3-3-138, Sugimoto, Sumiyoshi-ku, Osaka, 5588585}
}
\email{matsuzaway@omu.ac.jp}

\address{
    \textsc{School of Mathematical Sciences, Shanghai Key Laboratory of PMMP}\endgraf
    \textsc{East China Normal University, 500 Dongchuan Road, Shanghai 200241, People's Republic of China;}\endgraf
	\textsc{Korea Institute For Advanced Study,
		Seoul 02455, Republic of Korea}
}
\email{smeng@math.ecnu.edu.cn; ms@u.nus.edu}

\address{
\textsc{National University of Singapore, Singapore 119076, Republic of Singapore;}\endgraf
\textsc{National Fisheries University, 2-7-1 Nagata-Honmachi,
Shimonoseki 759-6595, Japan}
}
\email{shibata@fish-u.ac.jp; mattash@nus.edu.sg}

\address{National University of Singapore, Singapore 119076, Republic of Singapore}
\email{matzdq@nus.edu.sg}

\begin{document}

\begin{abstract}
Given a surjective endomorphism $f: X \to X$ on a projective variety over a number field,
one can define the arithmetic degree $\alpha_f(x)$ of $f$ at a point $x$ in $X$.
The Kawaguchi--Silverman Conjecture (KSC) predicts that any forward $f$-orbit of a point $x$ in $X$ at which the arithmetic degree $\alpha_f(x)$ is strictly smaller than the first dynamical degree $\delta_f$ of $f$ is not Zariski dense.
We extend the KSC to sAND (= small Arithmetic Non-Density) Conjecture that the locus $Z_f(d)$ of all points of small arithmetic degree is not Zariski dense, and verify this sAND Conjecture for endomorphisms on projective varieties including surfaces, HyperK\"ahler varieties, abelian varieties, Mori dream spaces, simply connected smooth varieties admitting int-amplified endomorphisms, smooth threefolds admitting int-amplified endomorphisms, and some fibre spaces. We show the equivalence of the sAND Conjecture and another conjecture on the periodic subvarieties of small dynamical degree; we also show the close relations between the sAND Conjecture and the Uniform Boundedness Conjecture of Morton and Silverman on endomorphisms of projective spaces and another long standing conjecture on Uniform Boundedness of torsion points in abelian varieties.
\end{abstract}

\maketitle

\tableofcontents

\section{Introduction}\label{sec1}

Let $f: X \to X$ be a surjective endomorphism of a geometrically irreducible projective variety $X$ defined over a number field $K$. Fix an algebraic closure $\overline{K}$ of $K$.
One may consider two dynamical invariants: the (first) {\it dynamical degree} 
$$\delta_f= \lim_{n \to +\infty} ((f^n)^*H \cdot H^{\dim (X) -1})^{1/n}$$
of $f$ and the {\it arithmetic degree} 
$$\alpha_f(x)=\lim_{n \to +\infty} h_H(f^n(x))^{1/n}$$
of $f$ at $x \in X(\overline K)$.
Here $H$ is any ample divisor on $X$ and $h_{H}$ is (the maximum of $1$ and) the Weil height function associated with $H$.
See \ref{n:2.1} for some basic properties of these invariants. 
The dynamical degree reflects the geometric complexity of iterations of $f$.
On the other hand, the arithmetic degree reflects the arithmetic complexity of the given (forward) $f$-{\it orbit}
$O_f(x) :=\{ x, f(x), f^2(x), \ldots \}$.

One naturally likes to compare these two invariants.
By Kawaguchi--Silverman \cite{KS16a}, the arithmetic degree at any point is less than or equal to the dynamical degree (see also \cite{Mat20}).
Then, one wonders when the arithmetic degree at a point attains (its upper bound) the dynamical degree.
Kawaguchi and Silverman proposed (see \cite{KS16a}):

\begin{Conjecture}\label{conj_ks} ( {\bf Kawaguchi-Silverman Conjecture, KSC in short})
Let $f: X \to X$ be a surjective endomorphism of a projective variety $X$ defined over a number field $K$.
Let $x \in X(\overline K)$ such that $ \alpha_{f}(x)< \delta_{f}$.
Then the orbit $O_f(x)$ is not Zariski dense in
$X_{\overline{K}} := X \times_K \overline{K}$.
\end{Conjecture}

Although the original conjecture in \cite{KS16a} has also been formulated for rational maps, we consider only endomorphisms in this paper (see Remark \ref{rem_prev}).

Given a surjective endomorphism $f: X \to X$ on a projective variety $X$ over a number field $K$,
let
$Z_f \subseteq X(\overline K)$
be the set of points $x \in X(\overline K)$ where $\alpha_f(x) <\delta_f$.
Conjecture \ref{conj_ks} predicts that the orbit $O_f(x)$ is not dense for any $x \in Z_f$.
Naturally, it is important for us to characterize the set $Z_f$.

We recall the {\it Northcott finiteness property} (cf.~\cite[Theorem B.3.2.(g)]{HS00} or \cite[Theorem 2.6]{Lan83}).
Let $X$ be a projective variety over a number field $K$ and $H$ an ample Cartier divisor on $X$.
Take a height function $h_H$ associated to $H$.
Then the Northcott finiteness property asserts that the set
$$\{ x \in X(\overline K) \mid [K(x):K] \leq d, h_H(x) \leq M \}$$
is finite for any given $d>0$ and $M>0$.
This suggests the non-density of a subset of $Z_f$ when we impose a bound on the degree $[K(x):K]$ for points $x$ in $Z_f$.
In view of this observation, we are going to propose Conjecture \ref{conj_zf}.
But before that, it is better for us to introduce the following definitions for convenience.

\begin{defn}\label{defn3.1.2}
Let $X$ be a projective variety and $f$ a surjective endomorphism on $X$ over a number field $K$.
Let $d > 0$. 
Set
$$X(d)=X(K,d)=\{ x \in X(L) \mid K \subseteq L \subseteq \overline K,\ [L:K] \leq d\},$$
$$Z_{f}=Z_f( \overline{K})=\{ x \in X(\overline{K}) \mid  \alpha_f(x) < \delta_f \},$$
$$Z_{f}(d)=Z_f(K, d)=Z_f( \overline{K})\cap X(K,d).$$
{\it We sometimes drop $ \overline{K}$ or $K$ when it is clear from the context.}
\end{defn}

\begin{Conjecture}\label{conj_zf} ({\bf sAND = small Arithmetic Non-Density})
Let $f:X\to X$ be a surjective endomorphism of a projective variety $X$ defined over a number field $K$.
Then the set $Z_f(K, d)$ is not Zariski dense in $X_{\overline{K}} = X \times_K \overline{K}$
for any positive constant $d>0$.
\end{Conjecture}

\begin{rem}\label{rem_prev}\
\begin{enumerate}
\item We cannot claim the non-density of $Z_f$. Here is a simple example.
Let the morphism $f: \mathbb P^1 \to \mathbb P^1$ be given by $f(x)=x^2$ in affine coordinates.
Then the set
$$f^{-\infty}(1):=\bigcup_{n=0}^\infty f^{-n}(1)
=\bigcup_{n=0}^\infty \{ x \in \overline{\mathbb Q} \mid x^{2^n}=1 \}$$
consists of $f$-preperiodic points and is infinite.
Clearly, any $f$-preperiodic point has arithmetic degree one while $\delta_f=2$, so
$f^{-\infty}(1) \subseteq Z_f$.
Hence $Z_f$ is a dense subset of $\mathbb P^1$.

\item  In \cite{Shi19}, the third author defined a similar notion $Z_f(K)$ and conjectured the non-density of it (\cite[Conjecture 1.3]{Shi19}).
It is the set of points at which the ample height does not grow maximally for iterations of $f$.
However, it is easier for us to handle the arithmetic degree than the height growth.
In addition, we can consider arithmetic degrees for rational self-maps or endomorphisms on non-projective varieties.
So we propose the present formulation in this paper.
\item  The sAND Conjecture \ref{conj_zf} implies the Kawaguchi--Silverman conjecture (Conjecture \ref{conj_ks}) for endomorphisms.
Indeed, for any $x\in Z_f( \overline{K})$, we can take some $d>0$ such that the orbit $O_f(x) \subseteq Z_f(K, d)$.
So the sAND Conjecture \ref{conj_zf} is a generalization of the Kawaguchi--Silverman conjecture \ref{conj_ks}.

\item 
The sAND Conjecture \ref{conj_zf} is further generalized to another (equivalent) Conjecture \ref{Conjecture2} on periodic subvarieties of small dynamical degree.
See Section \ref{Sect_ext} for details.

\item 
Let $K \subseteq L$ ($\subseteq \overline{K}$) be a finite extension.
Then sAND Conjecture \ref{conj_zf} for $f \colon X \longrightarrow X$ is equivalent to that of $f_{L} \colon X_{L} \longrightarrow X_{L}$.
Indeed, it is easy to see that we have
\begin{align*}
Z_{f_{L}}(L, d) \subset Z_{f}(K, [L:K]d) \quad \text{and} \quad Z_{f}(K, d) \subset Z_{f_{L}}(L, d)
\end{align*}
and the equivalence follows.

\item 
In the sAND Conjecture \ref{conj_zf}, it is important to assume that $f$ is a well-defined morphism.
Indeed,
Call and Silverman \cite[Proposition 13 (a)]{CS18} mentioned Junyi Xie's example:
\begin{align*}
f \colon \PP^{2} \times \PP^{1} &\dashrightarrow \PP^{2} \times \PP^{1}\\
((x:y:z),t) &\mapsto (((x+z)(x-tz)+(x-z)y : (x-z)y : (x-tz)z),t)
\end{align*}
where $t$ is the inhomogeneous coordinate of $\PP^{1}$.
This dominant rational map has the following properties:
\begin{enumerate}
\item $\pr_{2} \circ f = \pr_{2}$;
\item $ \delta_{f} = 2$;
\item for $t \in \mathbb{A}^{1}(\Q) \subset \PP^{1}(\Q)$, let $f_{t}=f|_{\pr_{2}^{-1}(t)} \colon \PP^{2} \dashrightarrow \PP^{2}$.
Then $ \delta_{f_{t}} < 2 = \delta_{f}$ for all $t \in \Z_{>0}$.
\end{enumerate}
For $t \in \Z_{>0}$, we can easily show that any point of
\[
\{ (x:y:z) \in \PP^{2}(\Q) \mid y,z,x-z,x-tz>0\}
\]
has well-defined $f_{t}$-orbits. Since $Z_{f}(\Q, d)$ must contain every such point with $t$ varying over positive integers,
it is Zariski dense in $\PP^{2} \times \PP^{1}$.
\end{enumerate}
\end{rem}

Another motivation for us to extend KSC Conjecture \ref{conj_ks} to sAND Conjecture \ref{conj_zf} is in order to cover the case of an endomorphism on a fibre space which fixes fibres set-theoretically.
Such a surjective endomorphism has no dense orbit,
so the KSC Conjecture \ref{conj_ks} vacuously holds.
However, the sAND Conjecture \ref{conj_zf} is still non-trivial in this case.
In fact, we will prove sAND Conjecture \ref{conj_zf} for some special fibrations in Section \ref{sec_fib}.

Theorems \ref{ThmA}, \ref{ThmB}, and \ref{ThmC} below are our main results;
see also Theorem \ref{thm_pol} for polarized endomorphisms, Theorem \ref{thm_projbdl} for endomorphisms of projective bundles fixing the base pointwise,
and Theorem \ref{thm_kappa2} for non-invertible surjective endomorphisms on smooth projective threefolds of non-negative Kodaira dimension.

For non-isomorphic surjective endomorphisms, we focus on the essential case when the Kodaira dimension is non-positive.
By taking the maximal rationally connected fibration, which can be chosen to
be $f$-equivariant, we see that the building blocks of varieties are rationally connected
varieties and varieties of Kodaira dimension zero; further, each variety of the last type, if smooth and minimal, has an (equivariant) Beauville--Bogomolov \'etale cover by a product of Abelian varieties, Hyperk\"ahler varieties and Calabi--Yau varieties.
Our result below shows that Conjecture \ref{conj_zf} holds in many cases, and it covers the three types of building blocks: Abelian varieties, Hyperk\"ahler varieties and
Rationally connected varieties (but not Calabi--Yau varieties).

We would like to remark that Theorem \ref{thm-tir-nn2}, concerning the non-existence of Case TIR, has its own interest. 
It employs several new techniques in birational geometry and dynamical system. The main ingredient there is to study the minimal model program after a base change which is induced by a totally invariant multi-section of the natural fibration there.
Our Theorem \ref{ThmA}(6) below is a special case showing the usefulness of the techniques. 

\begin{theorem}\label{ThmA}
The sAND Conjecture \ref{conj_zf} holds for every surjective endomorphism on any projective variety $X$ which fits one of the following cases.
\begin{itemize}
\item[(1)] $X_{ \overline{K}}$ is a Mori dream space (cf.~Theorem \ref{thm_mds}).
\item[(2)] $\dim(X)\le 2$ (cf.~Theorems \ref{thm_surfauto} and \ref{thm_surfendo}).
\item[(3)] $X_{ \overline{K}}$ is a Hyperk\"ahler variety (cf.~Theorem \ref{thm_hyper}).
\item[(4)] $X_{ \overline{K}}$ is a ($Q$)-abelian variety (cf.~Corollary \ref{cor_qab}).
\item[(5)] $X_{ \overline{K}}$ is a smooth rationally connected projective variety admitting an int-amplified endomorphism (cf.~Theorem \ref{thm_iamp-src} and Corollary \ref{cor-iamp} for more general statement including the case when $X_{ \overline{K}}$ is simply connected).
\item[(6)] $X_{ \overline{K}}$ is a smooth projective threefold admitting an int-amplified endomorphism (cf.~Theorem \ref{thm-tir-nn2} and Corollary \ref{cor-iamp}).
\end{itemize}
\end{theorem}

Let us recall the uniform boundedness conjecture due to Morton--Silverman.

\begin{Conjecture} \label{conj_ubc} ({\bf UBC = Uniform Boundedness Conjecture} for $\PP^N$, {\cite{MS94}})
Fix $r \in \mathbb Z_{\geq 2}$, $N \in \mathbb Z_{\geq 1}$, and $d \in \mathbb Z_{\geq 1}$.
Then there exists a positive constant $C=C(r,N,d)>0$ such that the following holds:
for any finite extension $K \subseteq L$ ($\subseteq \overline{K}$) with $[L:K] \leq d$ and any morphism $\phi: \mathbb P^N \to \mathbb P^N$ which is defined over $L$ with $\deg \phi=r$,
the number of $\phi$-preperiodic $L$-rational points of $\mathbb P^N$ is at most $C$.
\end{Conjecture}

\begin{rem}
The constant in Conjecture \ref{conj_ubc}  and that in Conjecture \ref{conj_tor} are usually stated that
they depend on the extension degree of $L$ over $\Q$.
Here, we state it in a slightly different form so that they are more compatible with our other settings.
These forms where $K$ varies are equivalent to the standard ones.
\end{rem}

We show the equivalence of the above Uniform Boundedness Conjecture
and a special case of a relative version of the sAND Conjecture \ref{conj_relsand}: the non-density of $Z_f(d)$ for a family of polarized endomorphisms,
which suggests the importance and naturality of conjecturing the non-density of the locus $Z_f(d)$ of small arithmetic degree.

\begin{theorem}\label{ThmB}
The following statements are equivalent.
\begin{itemize}
\item[(1)]
The Uniform Boundedness Conjecture \ref{conj_ubc}.
\item[(2)]
Let $\pi: X \to S$ be a projective morphism of varieties and $f: X \to X$ a surjective morphism satisfying $\pi \circ f=\pi$.
Assume that there is a $\pi$-ample divisor $H$ on $X$ such that $f^*H \sim_\pi rH$ for some $r>1$.
Let $d>0$.
Then there is a constant $C>0$ (depending on $d$ as well as $\pi$ and $f$) such that $|Z_{f|_{X_s}}(K,d)| < C$ for any $s \in S(K, d)$.
\item[(3)]
Let $\pi: X \to S$ be a projective morphism of quasi-projective varieties and $f: X \to X$ a surjective morphism satisfying $\pi \circ f=\pi$.
Assume that there is a $\pi$-ample divisor $H$ on $X$ such that $f^*H \sim_\pi rH$ for some $r>1$.
Let $d>0$.
Then $Z_f(K, d)$ is not dense.
\end{itemize}
\end{theorem}

Our sAND conjecture is also related to another long-standing conjecture
which is completely solved only in dimension one by Merel \cite[Corollaire]{Mer96}.

\begin{conj}\label{conj_tor}
({\bf UBC = Uniform Boundedness Conjecture} for Torsion Points)
Take any $N \in \mathbb Z_{\geq 1}$ and $d>0$.
Then there is a constant $C=C(N,d)>0$ such that the number of torsion points in $A(K, d)$ is at most $C$ for any abelian variety $A$ of dimension $N$ defined over $K$.
\end{conj}

\begin{rem}\label{r:UBC^2}\
\begin{itemize}
\item[(1)]
Conjecture \ref{conj_ubc} implies
Conjecture \ref{conj_tor} (cf.~\cite[Corollary 2.4]{Fak03}).

\item[(2)]
Conjecture \ref{conj_tor} is usually stated as follows:

($*$) \,\, \textit{
Take any $N \in \mathbb Z_{\geq 1}$ and $d>0$.
Then there is a constant $C=C(N,d)>0$ such that the order of $\Tor(A(L))$ is
at most $C$ for any abelian variety $A$ of dimension $N$ over $K$ and
any finite extension $K \subseteq L$ ($\subseteq \overline{K}$) with $[L:K] \leq d$.
}

In fact, ($*$) is equivalent to Conjecture \ref{conj_tor}.
Precisely, assuming ($*$), the integer $m=([C]+1)!$ depending only on $N$ and $d$ annihilates all torsion points of $A(L)$ for any abelian variety $A$ of dimension $N$ and
any finite extension $K \subseteq L$ ($\subseteq \overline{K}$)
with $[L:K] \leq d$.
So $m$ annihilates all torsion points contained in $A(d)$.
This implies that the number of torsion points in $A(d)$ is at most $m^{2N}$.

\item[(3)]
Conjecture \ref{conj_tor} for $N=1$ i.e.~the elliptic curve case is known (cf.~\cite[Corollaire]{Mer96}):
For any $d>0$, there is a constant $C=C(d)>0$ such that the number of torsion points in $E(d)$
is at most $C$ for any elliptic curve $E$ defined over $K$.

\end{itemize}
\end{rem}

As an evidence of our sAND Conjecture (\ref{conj_zf} or \ref{conj_relsand}), we show that
Conjecture \ref{conj_tor} (or a weaker assumption) implies Conjecture \ref{conj_relsand} for abelian fibrations.

\begin{thm}\label{ThmC}
Let $\pi: X \to Y$ be an abelian fibration of quasi-projective varieties and
$f: X \to X$ a surjective morphism such that $\delta_f>1$ and $\pi \circ f=\pi$.
Assume either Conjecture \ref{conj_tor} or weakly that there is a Zariski dense open subset $U\subseteq Y$ such that
for any $d>0$ and $d'>0$, there is a constant $C=C(d,d')>0$ such that
the number of torsion points in $X_y(K, d)$, where $X_y$ is the fibre over $y \in U(K, d')$, is at most $C$.
Then for any $d>0$, $Z_f(K, d)$ is not dense in $X_{\overline{K}}$.
\end{thm}

\begin{ack}
The first, third and last authors are supported by a JSPS Overseas Research Fellowship, a Research Fellowship of NUS, and an ARF of NUS, respectively. 
The second author is supported in part by Science and Technology Commission of Shanghai Municipality (No. 22DZ2229014) and by a Research Fellowship of KIAS (MG075501).
The authors would like to thank Fei Hu for many helpful suggestions, Shu Kawaguchi for kindly pointing out Remark \ref{r:UBC^2} (1), and colleagues for pointing out that our proof of Theorem \ref{ThmC} actually used an assumption weaker than (initially stated) Conjecture \ref{conj_tor}.
The authors would also like to thank the referee for valuable suggestions to improve this paper.
\end{ack}

\section{Preliminaries}\label{sec_prelim}

In this section, we prepare some notions and lemmas needed.

\begin{setup}\label{n:2.1} {\bf Notation, Convention and Terminology}
\begin{itemize}
$ \, $
\item
Throughout this article, we work over a fixed number field $K$ unless otherwise stated.
We fix an algebraic closure $\overline K$ of $K$.
The base change to $ \overline{K}$ of some object $A$ defined over $K$ is denoted by $A_{ \overline{K}}$.

\item A \textit{variety} means a geometrically integral separated scheme of finite type over $K$.

\item An \textit{endomorphism on $X$} is a morphism over $K$ from $X$ to itself.

\item Let $X$ be a projective variety and $f$ a surjective endomorphism on $X$.
\begin{itemize}
\item[(1)]
The (\textit{forward}) \textit{$f$-orbit} of a point $x$ in $X(\overline{K})$ is
the set $O_f(x) = \{f^s(x) \, | \, s \ge 0\}$.

\item[(2)]
A point $x$ in $X(\overline{K})$ is \textit{$f$-periodic} if
$f^n(x)=x$ for a positive integer $n$.
$\Per(f)$ denotes the set of $f$-periodic $\overline K$-rational points of $X$.
\item[(3)]
A point $x$ in $X(\overline{K})$ is \textit{$f$-preperiodic} if
$f^k(x)$ is $f$-periodic for a positive integer $k$.
$\Prep(f)$ denotes the set of $f$-preperiodic $\overline{K}$-rational points of $X$.

We remark that $x$ is $f$-preperiodic if and only if $O_f(x)$ is finite.
Further, if $f$ is an automorphism, then $x$ is $f$-preperiodic if and only if
$x$ is $f$-periodic.
\item[(4)]
A closed subset $V \subseteq X$ is \textit{$f$-invariant} (resp.~$f^{-1}$-{\it invariant}) if $f(V) \subseteq V$ (resp.~$f^{-1}(V) \subseteq V$).
Similarly $V$ is \textit{$f$-periodic} if $f^n(V) \subseteq V$ for some $n>0$ and
\textit{$f$-preperiodic} if $f^{n+k}(V) \subseteq f^k(V)$ for some $n,k>0$.
\end{itemize}

\item Let $X$ be a projective variety and $f$ a surjective endomorphism on $X$.
The \textit{(first) dynamical degree} $\delta_{f}$ of $f$ is defined to be the (first) dynamical degree of $f_{ \overline{K}}$.
Namely, it is the following limit, with $H$ any ample divisor
$$\delta_f= \lim_{n \to +\infty} ((f^n)^*H \cdot H^{\dim (X) -1})^{1/n}.$$

\item
Let $f: X \to X$ be a surjective endomorphism on a projective variety.
Let $Z \subseteq X_{\overline K}$ be an $f$-preperiodic subvariety.
\begin{itemize}
\item[(1)]
Suppose $f^{m+k}(Z)=f^k(Z)$ for some $m,k>0$.
Define the {\it dynamical degree} of $f$ along $Z$ as
$$\delta_{f|_Z}=(\delta_{f^m|_{f^k(Z)}})^{\frac{1}{m}}.$$
Clearly this definition is independent of $m$ and $k$.
If $Z$ is $f$-invariant, then this invariant coincides with the dynamical degree of $f|_Z$.
\item[(2)]
We call $Z$ a \textit{$f$-preperiodic subvariety of small dynamical degree}
if $\delta_{f|_Z} < \delta_f$.
Moreover we say $Z$ is a \textit{maximal $f$-preperiodic subvariety of small dynamical degree} if $Z$ is maximal among all of the $f$-preperiodic subvarieties of small dynamical degree.
\end{itemize}

\item The symbols $\sim$ (resp.~$\sim_{\mathbb Q}$, $\sim_{\mathbb R}$) and
$\equiv$ (resp.~$\equiv_w$) mean
the linear equivalence (resp.~$\mathbb Q$-linear equivalence,
$\mathbb R$-linear equivalence) and the numerical equivalence (resp.~weak numerical equivalence) on divisors (cf.~\cite[Definition 2.2]{MZ18}).

\item
A surjective morphism $f: X \to X$ is $q$-{\it polarized} (or just polarized) if
$f^*H \sim qH$ for some ample divisor $H$ and integer
$q > 1$ (or equivalently if $f^*B \equiv qB$
for some big $\R$-divisor $B$; cf.~\cite[Proposition 3.6]{MZ18}).

\item
A surjective morphism $f: X \to X$ is {\it int-amplified} if
$f^*L-L=H$ for some ample divisors $L$ and $H$
(or equivalently if $f^*L-L=H$ for some big $\R$-divisors $L$ and $H$; cf.~\cite[Theorem 1.1]{Men20}).

\item
A finite surjective morphism is {\it quasi-\'etale} if it is
\'etale in codimension one.

\item
A normal projective variety $X$ over $ \overline{K}$ is {\it $Q$-abelian}, if there is a finite quasi-\'etale
morphism $A \to X$ from an abelian variety $A$.
A normal projective variety $X$ over $K$ is $Q$-abelian if so is $X_{ \overline{K}}$.

\item For a projective variety $X$ over $ \overline{K}$, $\NS(X):=\Pic(X)/\Pic^0(X)$ denotes
the {\it N\'eron-Severi} group of $X$.
Let $\mathbb K= \mathbb Q, \mathbb R$ or $\mathbb C$.
Set $\NS(X)_{\mathbb K} := \NS(X) \otimes_{\Z} {\mathbb K}$.
Set $\N^1(X) :=\NS(X)_{\mathbb R}$ and
$\rho(X)=\dim_{\mathbb R} \N^1(X)$.
The number $\rho(X)$ is called the \textit{Picard number of $X$}.
Denote by $\Nef(X)$ the cone of nef divisors in $\N^1(X)$.
Denote by $\PE^1(X)$ the cone of pseudo-effective divisors in $\N^1(X)$.

\item Let $\mathbb K= \mathbb Q, \mathbb R$ or $\mathbb C$.
For a $\mathbb K$-linear endomorphism $\phi: V \to V$ on a $\mathbb K$-vector space $V$,
$\rho(\phi)$ denotes the {\it spectral radius} of $f$, that is,
the maximum of absolute values of eigenvalues (in $\C$) of $\phi$.

\item
Let $f: X \to X$ be a surjective endomorphism of a projective variety $X$. Then it is known that
the first dynamical degree $\delta(f)$ is equal to
$\rho(f^*|_{\NS(X)_{\C}})$.

\item Let $f$, $g$ and $h$ be  $\mathbb R$-valued functions on a domain.
Denote $f = g + O(h)$ if there is a positive constant $C$ such that
$|f-g| \leq C |h|$.
In particular,
denote $f=g + O(1)$ if there is a positive constant $C$ such that
$|f-g| \leq C$.

\item Let $X$ be a projective variety.
For an $\mathbb R$-Cartier $\mathbb R$-divisor $D$ on $X$,
we can attach a function $h_D: X(\overline{K}) \to \mathbb R$ called logarithmic Weil height function.
The function $h_D$ is called the \textit{height function associated to $D$}.
Note that to construct $h_{D}$ from $D$, we need additional information, for example, an adelic metric on the associate line bundle if $D$ is integral,
but the resulting function is unique up to bounded functions.
For definition and properties of height functions, see e.g.~\cite[Part B]{HS00} or \cite[Chapter 3]{Lan83}.

\item Let $X$ be a normal projective variety and $f$ a surjective endomorphism on $X$.
Fix an ample height function $h_H \geq 1$.
Then, for every $x \in X(\overline{K})$, the limit
$$\alpha_f(x)=\lim_{n \to +\infty} h_H(f^n(x))^{1/n}$$
exists and is independent of the choice of ample height function;
we call $\alpha_f(x)$ the \textit{arithmetic degree of $f$ at $x$}.
(cf.~\cite{KS16a}, \cite[Theorem 3 (a)]{KS16b} for details).

\item Let $f:X \to X$ be a surjective endomorphism on a projective variety $X$ and $D$ an $\mathbb R$-Cartier divisor on $X$.
Take a height function $h_D$ associated to $D$.
Fix an ample Cartier divisor $H$ on $X$ and an associated height function $h_H \geq 1$.
\begin{itemize}
\item[(1)]
Assume that $f^*D \sim_{\mathbb R} \lambda D$ with $\lambda >1$.
Then, for any $x \in X(\overline K)$, the limit
$$\hat h_{D,f}(x)=\lim_{n \to +\infty} \frac{h_D(f^n(x))}{\lambda^n}$$
converges and satisfies
$\hat h_{D,f}=h_D+O(1)$ (cf.~\cite[Theorem 1.1]{CS93}).
\item[(2)]
Assume that $f^*D \equiv \lambda D$ with $\lambda >\sqrt{\delta_f}$.
Then, for any $x \in X(\overline K)$, the limit
$$\hat h_{D,f}(x)=\lim_{n \to +\infty} \frac{h_D(f^n(x))}{\lambda^n}$$
converges and satisfies
$\hat h_{D,f}=h_D+O(\sqrt{h_H})$ (cf.~\cite[Theorem 5]{KS16a}).
\end{itemize}

\item Given a variety $X$ and $d>0$, set
$$X(d)=\bigcup_{K \subseteq L \subseteq \overline K, \, [L:K] \leq d} X(L).$$

\item For an abelian group $G$,
$\Tor(G)$ denotes the set of torsion elements of $G$.

\item Let $A$ be an abelian variety over $ \overline{K}$.

\begin{itemize}
\item[(1)] For a Cartier divisor $D$ on $A$, let $\varphi_{D} \colon A \longrightarrow \hat{A} ; x \mapsto [T_{x}^{*}D-D]$,
where $\hat{A}$ is the dual abelian variety and $T_{x}$ is the translation by $x$.
Note that $D \in \Pic^{0}(A)$ if and only if $\varphi_{D}=0$ (this is also adopted by Mumford as the definition of $\Pic^{0}(A)$ in his book).
We get an $\R$-linear map $\NS(A)_{\R} \longrightarrow \Hom(A, \hat{A}) {\otimes}_{\Z}\R$.
Note that if $D$ is ample, then $\varphi_{D}$ is an isogeny.

\item[(2)] For an ample divisor $H$ on $A$,
set $$\Phi^{H} \colon \NS(A)_{\R} \longrightarrow \End(A)_{\R} ; D \mapsto \Phi^{H}_{D}=\varphi_{H}^{-1}\circ \varphi_{D}.$$

\item[(3)]
Denote the {\it Rosati involution} with respect to an ample divisor $H$ by
$$i_{H} \colon \End(A)_{\R} \longrightarrow \End(A)_{\R} ; \alpha \mapsto \varphi_{H}^{-1}\circ \hat{\alpha}\circ \varphi_{H}$$ where $\hat{\alpha}$ is the dual of $\alpha$ in $\End(\hat{A})_{\R}$.
\end{itemize}

\item For a generically finite surjective morphism
$\pi: X \to Y$ between varieties over $\overline{K}$, denote by
$\Exc(\pi)$ the {\it $\pi$-exceptional locus}, i.e., the union of curves in $X$ contracted by $\pi$.
\end{itemize}
\end{setup}

The following lemma follows from the product formula (cf.~\cite{DN11} and \cite{Tru20}).

\begin{lem}\label{lem_dyn_deg_rel}
Let $\pi: X \dashrightarrow Y$ be a dominant rational map between projective varieties,
and $f : X \to X$, $g: Y \to Y$ surjective morphisms such that
$\pi \circ f = g \circ \pi$. Then we have:
\begin{itemize}
\item[(1)]
$\delta_f \ge \delta_g$.
\item[(2)]
If $\dim (X) = \dim (Y)$, then $\delta_f = \delta_g$.
\end{itemize}
\end{lem}

The following two lemmas enable us to replace a given endomorphism by its positive power or direct summands, respectively.

\begin{lem}\label{lem_end_power}
Let $f:X\to X$ be a surjective endomorphism of a projective variety.
Take a positive integer $N$.
Then $Z_f=Z_{f^N}$.
In particular,
Conjecture \ref{conj_zf} holds for $f$ if and only if it holds for $f^N$.
\end{lem}

\begin{proof}
The lemma follows from that $\delta_{f^N}=\delta_f^N$ and $\alpha_{f^N}(x)=\alpha_f^N(x)$.
\end{proof}

\begin{lem}\label{lem_prod}
Let $X, Y$ be projective varieties and $f, g$ surjective endomorphisms on $X, Y$, respectively.
Then
$$
Z_{f \times g}=
\begin{cases}
Z_f \times Y \ \ \ \mathrm{if}\ \delta_f > \delta_g, \\
Z_f \times Z_g
\ \ \ \mathrm{if}\ \delta_f = \delta_g, \\
X \times Z_g \ \ \ \mathrm{if}\ \delta_f < \delta_g.
\end{cases}
$$
In particular, Conjecture \ref{conj_zf} holds if it holds for $f$ and $g$.
\end{lem}

\begin{proof}
We may assume that $\delta_f \geq \delta_g$, without loss of generality.
We can easily show (e.g. by the product formula) that $\delta_{f \times g}=\max ( \delta_f, \delta_g ) = \delta_f$.

Let $p: X \times Y \to X$, $q: X \times Y \to Y$ be the projections.
Take ample divisors $H_X, H_Y$ on $X,Y$, respectively.
Set $H=p^*H_X+q^*H_Y$. Then
\begin{align*}
\alpha_{f \times g}(x,y)
&=\lim_{n \to +\infty} h_{H}((f \times g)^n(x,y))^{1/n} \\
&=\lim_{n \to +\infty} \left \{ h_{H_X}(f^n(x))+h_{H_Y}(g^n(y)) \right \}^{1/n} \\
&=\max ( \alpha_f(x), \alpha_g(y) ).
\end{align*}
For the last equality, see \cite[Lemma 3.1]{San20}.
So $(x,y) \in Z_{f \times g}$ if and only if $\max $ $(\alpha_f(x), \alpha_g(y) )$ $<$
$\max ( \delta_f, \delta_g )$.

If $\delta_f > \delta_g$, then $(x,y) \in Z_{f \times g}$ if and only if $\alpha_f(x)<\delta_f$.
Hence $Z_{f \times g} =Z_f \times Y$.
If $\delta_f=\delta_g$, then $(x,y) \in Z_{f \times g}$ if and only if $\alpha_f(x) <\delta_f$ and $\alpha_g(y) < \delta_g$.
So $Z_{f \times g}=Z_f \times Z_g$.
\end{proof}

The following is easy but fundamental for the reduction to invariant subvarieties.

\begin{lem}\label{lem_subvar}
Let $f: X \to X$ be a surjective endomorphism on a projective variety $X$ and $W \subseteq X$ an $f$-invariant closed subvariety.
Then $\alpha_{f|_W}(x)=\alpha_f(x)$ for any $x \in W(\overline K)$.
\end{lem}

\begin{proof}
Let $\iota: W \to X$ be the inclusion.
For an ample divisor $H$ on $X$, the restriction $H|_W$ is also ample and $h_{H|_W}=h_H \circ \iota$.
So the assertion holds.
\end{proof}

The following lemma is useful when dealing with equivariant fibrations or covers.

\begin{lem}\label{lem_surj}
Let $X, Y$ be projective varieties and $f, g$ surjective endomorphisms on $X, Y$, respectively.
Let $\pi: X \to Y$ be a surjective morphism such that $\pi \circ f = g \circ \pi$. Then:
\begin{itemize}
\item[(1)]
$\alpha_f(x) \geq \alpha_g(\pi(x))$ for any $x \in X(\overline{K})$.
\item[(2)]
Assume that $\delta_f=\delta_g$.
Then
$\pi(Z_f) \subseteq Z_g$.
Moreover, Conjecture \ref{conj_zf} holds for $f$ if it holds for $g$.
\item[(3)]
Assume that $\pi$ is finite.
Then $\delta_f = \delta_g$, $\alpha_f(x) = \alpha_g(\pi(x))$,  and $\pi(Z_f)= Z_g$.
Moreover, Conjecture \ref{conj_zf} holds for $f$ if and only if it holds for $g$.
\end{itemize}
\end{lem}

\begin{proof}
(1)
Take ample divisors $H, A$ on $X, Y$ respectively such that $H - \pi^*A$ is ample.
Then we can take associated height functions $h_H, h_A \geq 1$ as satisfying
$Ch_H \geq  h_A \circ \pi$ for some $C>0$.
We have
$$h_A(g^n(\pi(x)))^{1/n}=h_A(\pi(f^n(x)))^{1/n} \leq \left(C h_H(f^n(x)) \right)^{1/n}.$$ Letting $n \to +\infty$, we obtain $\alpha_g(\pi(x)) \leq \alpha_f(x)$.

(2) Clearly $\pi(Z_f) \subseteq Z_g$ by (1).
Take any $d>0$.
Then $\pi(Z_f(d)) \subseteq Z_g(d)$.
So the non-density of $Z_g(d)$ implies the non-density of $Z_f(d)$.

(3) By Lemma \ref{lem_dyn_deg_rel}, $\delta_f = \delta_g$.
Since $\pi$ is finite, $\pi^*A$ is also ample.
This implies that $\alpha_f(x) = \alpha_g(\pi(x))$ for any $x \in X(\overline{K})$.
So $\pi(Z_f)=Z_g$.

Take any $d>0$.
Since $\pi$ is a finite cover, if we take $d' = d\deg \pi$, then
$\pi^{-1}(Y(d)) \subseteq X(d')$. 
So $\pi^{-1}(Z_g(d)) \subseteq Z_f(d')$.
Therefore the non-density of $Z_f(d')$ implies the non-density of $Z_g(d)$.
\end{proof}

Below is a description of the closure of an orbit.

\begin{lem}\label{lem_orb}
Let $f: X \to X$ be a surjective endomorphism on a projective variety $X$ over $ \overline{K}$ and $x \in X( \overline{K})$.
Then for some $s \ge 1$ and $t \ge 1$,
$$\overline{O_f(x)}=\{x\}\cup \{f(x)\} \cup \cdots \cup \{f^{t-1}(x)\}\cup \bigcup_{i=0}^{s-1}f^i(\overline{O_{f^s}(f^t(x))}),$$
which is a union of irreducible closed subsets and $\overline{O_{f^s}(f^t(x))}$ is $f^s$-invariant.
\end{lem}

\begin{proof}
Let $Z:=\overline{O_f(x)}$ which is $f$-invariant.
Since $Z$ has finitely many irreducible components,
we may assume $f^s(Z_1)=Z_1$ for for some $s>0$ and irreducible component $Z_1$ of $Z$.
Note that $y=f^t(x)\in Z_1$ for some $t>0$
(minimal).
Indeed, otherwise, we will have $Z=\overline{Z\backslash Z_1}$, a contradiction.
So $\overline{O_{f^s}(f^t(x))}\subseteq Z_1$.
Note that
$$O_f(x)=\{x\}\cup \{f(x)\} \cup \cdots \cup \{f^{t-1}(x)\}\cup \bigcup_{i=0}^{s-1}f^i(O_{f^s}(f^t(x))).$$
Since $f$ is a closed map,
$$Z=\{x\}\cup \{f(x)\} \cup \cdots \cup \{f^{t-1}(x)\}\cup \bigcup_{i=0}^{s-1}f^i(\overline{O_{f^s}(f^t(x))}).$$
Thus $\dim (\overline{O_{f^s}(f^t(x))})=\dim(Z_1)$, so $\overline{O_{f^s}(f^t(x))}=Z_1$ which is irreducible.
\end{proof}

\begin{lem}\label{lem_gen_fin}
Let $X,Y$ be projective varieties over $\overline{K}$ and $f,g$ surjective endomorphisms on $X,Y$, respectively.
Let $\pi: X \to Y$ be a generically finite surjective morphism such that $\pi \circ f=g \circ \pi$.
Then $\delta_f=\delta_g$ and $\alpha_f(x)=\alpha_g(\pi(x))$ for any $x \in X \setminus \Exc(\pi)$.
\end{lem}

\begin{proof}
By Lemma \ref{lem_dyn_deg_rel},
$\delta_f = \delta_g$.
Taking normalization, we may assume $X$ and $Y$ are normal by Lemma \ref{lem_surj}(3).
Taking the Stein factoriazation of $\pi$ (cf.~\cite[Lemma 5.2]{CMZ20}),
we may assume that $\pi$ is a birational morphism.
Then $\Exc(\pi)$ is $f^{-1}$-invariant by \cite[Lemma 7.3]{CMZ20}.

Let $x \in X \setminus \Exc(\pi)$.
By Lemma \ref{lem_orb}, replacing $x$ by an $f$-iteration and $f$ by a positive power, we may assume that $W=\overline{O_f(x)}$ is an irreducible closed $f$-invariant subset.
Note that $W$ is not contained in $\Exc(\pi)$.
So $\pi|_W: W \to \pi(W)$ gives a birational morphism.
Since $\alpha_f(x) = \alpha_{f|_W}(x)$ by Lemma \ref{lem_subvar},
replacing $\pi: X \to Y$ by $\pi|_W : W \to \pi(W)$, we may assume that $O_f(x)$ is Zariski dense in $X$.

Take an ample divisor $A$ on $Y$.
Then $\pi^*A$ is big, so $\pi^*A \sim H + E$ for some ample divisor $H$ and an effective divisor $E$ on $X$.
Take height functions $h_H, h_A \geq 1$.
Then we have $h_H \leq h_A \circ \pi +O(1)$ on $X \setminus \Supp(E)$.
So $h_H(f^n(x)) \leq h_A(g^n(\pi(x)))+O(1)$ if $f^n(x) \not\in \Supp(E)$.
We can take a subsequence $\{ f^{n_k}(x) \}_k \subseteq X \setminus \Supp(E)$ since $O_f(x)$ is dense.
Then
$$\alpha_f(x)=\lim_{k \to +\infty} h_H(f^{n_k}(x))^{1/n_k} \leq \lim_{k \to \infty} h_A(g^{n_k}(\pi(x)))^{1/n_k}=\alpha_g(\pi(x)).$$
This together with Lemma \ref{lem_surj} imply
$\alpha_f(x) = \alpha_g(\pi(x))$.
\end{proof}

We can extend Lemma \ref{lem_surj} to rational fibrations.

\begin{lem}\label{lem_rat}
Let $X, Y$ be projective varieties and $f, g$ surjective endomorphisms on $X, Y$, respectively.
Let $\pi: X \dashrightarrow Y$ be a dominant rational map such that $\pi \circ f = g \circ \pi$.
\begin{itemize}
\item[(1)]
Assume $\delta_f = \delta_g$. If Conjecture \ref{conj_zf} holds for $g$, then so does for $f$.
\item[(2)]
Assume that $\dim (X)=\dim (Y)$.
Then $\delta_f=\delta_g$. Moreover, Conjecture \ref{conj_zf} holds for $f$ if and only if it holds for $g$.
\end{itemize}
\end{lem}

\begin{proof}
Let $\Gamma$ be the graph of $\pi$, with the projections $p_1: \Gamma \to X$, $p_2: \Gamma \to Y$.
Let $h: \Gamma \to \Gamma$ be the restriction of $f \times g: X \times Y \to X \times Y$ to $\Gamma$.
Note that $\delta_h = \delta_f$ by Lemma \ref{lem_dyn_deg_rel}.
Note that $p_1$ is birational and $p_2$ is generically finite surjective when $\dim(X)=\dim (Y)$.
By Lemma \ref{lem_gen_fin}, we have $p_1^{-1}(Z_f)\backslash \Exc(p_1) = Z_h\backslash \Exc(p_1)$ and $p_2^{-1}(Z_g)\backslash \Exc(p_2) = Z_h\backslash \Exc(p_2)$ when $\dim(X)=\dim (Y)$.
Hence,
\begin{equation}\label{eq:Zdcomp}
p_1^{-1}(Z_f(d)) \, \subseteq \, Z_h(d'),
\hskip 1pc
p_2^{-1}(Z_g(d)) \, \subseteq \, Z_h(d') \,\,\,\,
({\text{\rm when}} \, \dim (X) = \dim (Y))
\end{equation}
with $d'$ depending on $d$.

(1)
If Conjecture \ref{conj_zf} holds
for $g$, it holds for $h$ by Lemma \ref{lem_surj}, and also for $f$ by $(\ref{eq:Zdcomp})$.

(2) We have $\delta_f = \delta_g$ by Lemma \ref{lem_dyn_deg_rel}. Assume $f$ satisfies Conjecture \ref{conj_zf}. Then $h$ satisfies Conjecture \ref{conj_zf} by Lemma \ref{lem_surj}.
So does $g$ by $(*)$ above.
\end{proof}

The following lemmas are used for the study of abelian fibrations in Section \ref{sec_fib}.

\begin{lem}[{cf.~\cite[Proposition 26]{KS16b}}]\label{lem_decomp}
Let $A$ be an abelian variety over an algebraically closed field of characteristic zero, $H$ an ample divisor and
$D$ a nef $\R$-divisor.
Then there exists some $\alpha \in \End(A)_{\R}$ such that $\Phi^{H}_{D}=i_{H}(\alpha)\circ \alpha$ and $i_{H}(\alpha)= \alpha$ (see \ref{n:2.1}).
\end{lem}

\begin{lem}\label{lem_spreadout} (cf.~\cite{FAG05})
Let $Y$ be a Noetherian scheme.
Let $\pi \colon X \longrightarrow Y$ be a smooth projective group scheme with geometrically integral fibres.
Let $D$ be a Cartier divisor on $X$.
Then the natural transformation
\begin{align*}
X(S) \longrightarrow \Pic_{(X/Y)_{(et)}}(S) ; \hskip 1pc x \mapsto [T_{x}^{*}D_{S}-D_{S}]  \ \
\end{align*}
(here $S$ runs over all locally Noetherian $Y$-schemes)
defines a morphism $X \longrightarrow {\rm \bf Pic}_{X/Y}$ which factors through ${\rm \bf Pic}^{0}_{X/Y}$.
Denote this morphism as $\varphi_{D} \colon X \longrightarrow {\rm \bf Pic}^{0}_{X/Y}$.
This construction commutes with base changes of $Y$, i.e., for any morphism $Z \longrightarrow Y$ from a Noetherian scheme $Z$,
we have the following commutative diagram:

\[
\xymatrix@C=70pt@R=30pt{
X {\times}_{Y} Z \ar[r]^{\varphi_{D} \times_{Y}Z} \ar[d]_{\id} & {\rm \bf Pic}^{0}_{X/Y}\times_{Y}Z \ar[d]^{\simeq}\\
X {\times}_{Y} Z \ar[r]_{\varphi_{D_{Z}}} & {\rm \bf Pic}^{0}_{X\times_{Y}Z/Z}.
}
\]

\end{lem}

\begin{rem}\label{rem_pic}
For Picard functors and Picard schemes, see \cite[Part 5]{FAG05}.
For relative $\Pic^{0}$, see \cite[Part 5, Proposition 9.5.20]{FAG05}.
\end{rem}

\begin{rem}\label{rem_phi}
In the setting of the above lemma, if $H$ is a
$\pi{-}$ample divisor on $X$, then $\varphi_{H}$ is a finite surjective morphism.
Thus we can define $\Phi^{H}_{D} \in \End(X/Y)_{\R}$ for an $\R$-Cartier divisor $D$ on $X$, and
also Rosati involution with respect to $H$, as in
\ref{n:2.1}.
\end{rem}

\section{Polarized endomorphisms, curves and Mori dream spaces}\label{subsec_pol}

\begin{thm}[cf.~{\cite{CS93}}, {\cite{KS14}}]\label{thm_pol}
Let $f: X \to X$ be a polarized endomorphism.
Then:
\begin{itemize}
\item[(1)]
$Z_f=\Prep(f)$.
\item[(2)]
For any $d>0$,
$X(d) \cap \Prep(f)$ is a finite set.
\end{itemize}
Hence Conjecture \ref{conj_zf} holds for every polarized endomorphism of any projective variety
and every surjective endomorphism of any projective curve.
\end{thm}

\begin{proof}
By assumption, $f^* H \sim dH$ for some ample divisor $H$ and $d>1$.
Then we have a canonical height function
$$\hat h_{H,f}(x)=\lim_{n \to +\infty} \frac{h_H(f^n(x))}{d^n}$$
which is a height function associated to $H$ (cf.~\cite{CS93}).
The Northcott property of $\hat h_{H,f}$ implies that
$\hat h_{H,f}(x)=0$ if and only if $x \in \Prep(f)$ for $x \in X(\overline K)$.
So $Z_f=(\hat h_{H,f}=0)=\Prep(f)$.
In addition, $\Prep(f) \cap X(d)$ is finite for any $d>0$ because of the Northcott finiteness property of $\hat h_{H,f}$.
\end{proof}

\begin{thm}[cf.~{\cite[Theorem 4.1]{Mat20a}}]\label{thm_mds}
Let $X$ be a normal projective variety such that $\N^1(X_{\overline{K}})_{\mathbb Q} = \Pic(X_{\overline{K}})_{\mathbb Q}$ and the nef cone of $X_{\overline{K}}$ is generated by finitely many semi-ample Cartier divisors.
Then $Z_f(d)$ is not Zariski dense in $X_{\overline K}$  for every surjective endomorphism $f$ on $X$ and any $d>0$.
In particular, Conjecture \ref{conj_zf} holds for all surjective endomorphisms on Mori dream spaces.
\end{thm}

\begin{proof}
We may assume that $\delta_f > 1$.
By the assumption, replacing $f$ by a positive power, $f^*D \sim \delta_f D$ for some base point free effective Cartier divisor $D\neq 0$.
Then $\hat h_{D,f}|_{Z_f}=0$ and hence $h_{D}|_{Z_f} = O(1)$.
For any $d>0$, we have
$Z_f(d) \subseteq \{ x \in X(d) \mid h_{D}(x) \leq M \}$
for some $M>0$.
By \cite[Corollary 2.3]{Shi19}), the set of right hand side and hence $Z_f(d)$ are not Zariski dense in $X_{\overline{K}}$.
\end{proof}

\section{Surfaces and Hyperk\"ahler varieties}\label{subsec_surfauto}

\begin{thm}[cf.~\cite{Kaw08}]\label{thm_surfauto}
Let $f: X \to X$ be an automorphism of positive entropy on a projective surface. Then we have:
\begin{itemize}
\item[(1)]
The number of $f$-periodic irreducible curves on $X_{\overline K}$ is finite.
\item[(2)]
Let $C_i$ ($1 \le i \le r$) be the $f$-periodic irreducible curves on $X_{\overline K}$.
Then $Z_f=\bigcup_{i=1}^r C_i \cup \Per(f)$.
\item[(3)]
For any $d>0$,
$X(d) \cap (\Per(f) \setminus (\bigcup_{i=1}^r C_i))$ is a finite set.
\end{itemize}
Hence Conjecture \ref{conj_zf} holds for every automorphism of any projective surface.

\end{thm}

\begin{proof}
Note that $\delta :=\delta_f$ ($>1$) is known to be a Salem number with $\delta_{f^{-1}}=\delta$.

First we consider the case when $X_{\overline{K}}$ is smooth.
We can find
nef $\mathbb R$-divisors $D^+, D^- \not\equiv 0$ on $X_{\overline{K}}$ such that $f^*D^+ \sim_{\mathbb R} \delta D^+$ and $f_* D^- \sim_{\mathbb R} \delta D^-$
(cf.~\cite[Lemma 3.8]{Kaw08} or \cite[Remark 5.11]{San17}).
Then $D=D^++D^-$ is nef and big by the Hodge index theorem.
Let $\{ C_{\lambda} \}_{\lambda \in \Lambda}$ be the family of all $f$-periodic irreducible curves.
Then $\Lambda$ is finite and we can take $a_\lambda >0$ for each $\lambda \in \Lambda$ such that $D-\sum_\lambda a_\lambda C_\lambda$ is ample (cf.~\cite[Proposition 1.3]{Kaw08}).

Set $E=\sum_\lambda a_\lambda C_\lambda$ and $A=D-E$.
Take canonical height functions
$$\hat h_{D^+,f}(x)=\lim_{n \to +\infty} \frac{h_{D^+}(f^n(x))}{\delta^n},\
\hat h_{D^-,f^{-1}}(x)=\lim_{n \to +\infty} \frac{h_{D^-}(f^{-n}(x))}{\delta^n}$$
associated to $D^+, D^-$, respectively.
Then
\begin{equation} \label{eq:htsum}
h_A+h_{E}=\hat h_{D^+,f}+ \hat h_{D^-,f^{-1}}+O(1).
\end{equation}
Substituting $f^n(x)$ to (\ref{eq:htsum}), we have
\begin{equation}\label{eq:htsum2}
h_A(f^n(x))+h_{E}(f^n(x))=\delta^n \hat h_{D^+,f}(x)+ \delta^{-n} \hat h_{D^-,f^{-1}}(x)+O(1).
\end{equation}
If $x \in Z_f$, then $\hat h_{D^+,f}(x)=0$ since otherwise we get $ \alpha_{f}(x) = \delta$ by 
the definition of arithmetic degree and the fact that $h_{D^{+}}$ is dominated by any ample height function (like $h_A$ here).
Note that $h_E$ is bounded from below outside $E$.
Thus if $x \in Z_{f} \setminus \Supp(E)$, then $\{ h_A(f^n(x)) \}_{n=0}^{+\infty}$ is upper bounded by (\ref{eq:htsum2}).
The Northcott property for $h_A$ implies $x \in \Per(f)$ since $f$ is an automorphism.
If $x \in \Supp(E)$, then $\alpha_f(x)=1$ since $f$ gives an automorphism on $\Supp(E)$ which is a union of curves and any automorphism on a curve has first dynamical degree one.
Therefore,
$$Z_f=\Supp(E) \cup \Per(f)=\bigcup_{\lambda \in \Lambda} C_\lambda \cup \Per(f).$$
Thus
the set $X(d) \cap (\Per(f) \setminus \bigcup_{\lambda \in \Lambda} C_\lambda)$ is finite for  any $d>0$ by the Northcott property of $h_A$. This proves the theorem for the case of smooth surfaces.

Now we consider the general case.
Take an $f$-equivariant resolution $\pi: \widetilde{X}_{\overline{K}} \to X_{\overline{K}}$ with an automorphism $\widetilde{f}:\widetilde{X}_{\overline{K}} \to \widetilde{X}_{\overline{K}}$ such that $\pi \circ \widetilde{f}=f \circ \pi$.
Extending $K$, we may assume that $\widetilde{X}$, $\pi$ and $\widetilde{f}$ are defined over $K$.
Lemma \ref{lem_surj}(2) and Lemma \ref{lem_gen_fin} imply that $\pi(Z_{\widetilde{f}}) \subseteq Z_f \subseteq \pi(Z_{\widetilde{f}} \cup \Exc(\pi))$.
Now $\Exc(\pi)$ is $f$-invariant, so it is a union of $f$-periodic curves.
Thus $Z_{\widetilde{f}} \cup \Exc(\pi)=Z_{\widetilde{f}}$ by Step 1.
Hence $\pi(Z_{\widetilde{f}})=Z_f$.

Let $\widetilde{C}_1, \ldots, \widetilde{C}_r$ be the $\widetilde{f}$-periodic irreducible curves on $\widetilde{X}_{\overline K}$.
Note that $\pi(\Exc(\pi))\subseteq \Per(f)$.
Then $\pi(\widetilde{C}_1),\ldots, \pi(\widetilde{C}_r)$ exhaust the $f$-periodic irreducible curves on $X_{\overline{K}}$ and
$$Z_f=\pi(Z_{\widetilde{f}})=\bigcup_{i=1}^r \pi(\widetilde{C}_i) \cup \Per(f).$$
Fix $d>0$.
By the case for smooth surfaces,
$\Per(\widetilde{f}) \setminus (\bigcup_{i=1}^r \widetilde{C}_i)$ is finite in $\widetilde{X}(d)$.
So $\Per(f) \setminus (\bigcup_{i=1}^r \pi(\widetilde{C}_i))$ (the image of $\Per(\widetilde{f}) \setminus (\bigcup_{i=1}^r \widetilde{C}_i)$) is also finite in $X(d)$.
\end{proof}

A projective variety $X$ over $K$ is {\it Hyperk\"ahler} if so is $X_{ \overline{K}}$.

\begin{thm}[cf.~{\cite[Theorem 1.2]{LS21}}]\label{thm_hyper}
Let $f: X \to X$ be a surjective endomorphism on a projective Hyperk\"ahler variety $X$.
Then $Z_f(d)$ is not Zariski dense in $X_{\overline{K}}$ for any $d>0$.
\end{thm}

\begin{proof}
We may assume  $\delta_f>1$.
By the ramification divisor formula, $f$ is \'etale
and hence an automorphism since a Hyperk\"{a}hler
variety has trivial $\pi_1^{\alg}(X_{ \overline{K}})$.
The proof is almost the same as that of Theorem \ref{thm_surfauto}.
Indeed, take nef $\mathbb R$-divisors $D^+, D^- \not\equiv 0$ on $X$ such that $f^*D^+ \sim_{\mathbb R} \delta_f D^+$ and $f_* D^- \sim_{\mathbb R} \delta_{f^{-1}} D^-$.
Then the Beauville--Bogomolov--Fujiki form enables us to show that $D^++D^-$ is nef and big (cf.~\cite[Lemma 3.8]{LS21}).
Write $D^++D^- \sim_{\mathbb R} A+E$ where $A$ is ample and $E$ is effective.
Then we have $Z_f \subseteq \Supp E \cup \Per(f)$.
The set $X(d) \cap (\Per(f) \setminus \Supp E)$ is finite for  any $d>0$ by the Northcott property of $h_A$.
So $Z_f(d)$ is not Zariski dense in $X_{\overline{K}}$.
\end{proof}

\begin{thm}[cf.~{\cite{MSS18}}, {\cite{MZ19a}}]\label{thm_surfendo}
Let $f: X \to X$ be a non-invertible surjective endomorphism on a projective surface.
Then $Z_f(d)$ is not Zariski dense in $X_{\overline{K}}$ for any $d>0$.
\end{thm}

\begin{proof}
We may assume $X_{\overline{K}}$ is normal by Lemma \ref{lem_surj}.
Since $\deg f>1$, $X_{\overline{K}}$ is log canonical by \cite[Theorem 2.8]{War90}.

If $K_{X_{\overline{K}}}$ is peudo-effective, then by \cite[Theorem 7.1.1]{Nak08}, after replacing $f$ by a suitable power and extending $K$,
there is a quasi-\'etale finite morphism $\phi: X' \to X$ with a lift $f': X' \to X'$ of $f$ (all are defined over $K$) such that
\begin{itemize}
\item[(1)]
$X'$ is an abelian surface; or
\item[(2)]
$X'$ is a product of an elliptic curve and a curve of genus $\geq 2$.
\end{itemize}

The first case is a special case of
Corollary \ref{cor_qab},
and the second case is a special case of Theorem \ref{thm_ellfib}, whose proofs are independent of the results of this theorem.

Assume that $K_{X_{\overline{K}}}$ is not pseudo-effective.
Then, after replacing $f$ by a suitable power and extending $K$, it is proved in \cite[Theorem 5.4]{MZ19a} that
\begin{itemize}
\item[(3)] $f$ is a polarized endomorphism; or
\item[(4)] There is a Fano contraction $\pi: X \to Y$ to a curve and $f$ induces $g: Y \to Y$ with $\delta_f=\delta_g$; or
\item[(5)] There is a finite surjective morphism $\tau: X \to \mathbb P^1 \times Y$ to the product of $\mathbb P^1$ and a curve $Y$ with endomorphisms $g:\mathbb P^1 \to \mathbb P^1$ and $h: Y \to Y$ such that $\tau \circ f = (g \times h) \circ \tau$.
\end{itemize}
These cases are verified by Theorem \ref{thm_pol}, Lemma \ref{lem_surj}, Lemma \ref{lem_prod}, respectively.
\end{proof}

Combining Theorems \ref{thm_surfauto} and \ref{thm_surfendo},
we obtain the following.

\begin{thm}\label{thm_surf}
Let $X$ be a projective surface and $f: X \to X$ a surjective endomorphism.
Then $Z_f(d)$ is not Zariski dense in $X_{\overline{K}}$ for any $d>0$.
\end{thm}

\section{Abelian varieties}\label{subsec_ab}

In this section, we first prove:

\begin{thm}[cf.~{\cite[Section 5]{MS20}}] \label{thm_ab}
Let $X$ be an abelian variety and $f: X \to X$ a surjective endomorphism
(which is not necessarily an isogeny)
with the first dynamical degree $\delta_f>1$.
Then there is
a proper abelian subvariety $B \subset X_{\overline K}$ and a point $p \in X(\overline{K})$
such that $B( \overline{K})+p$ is $f$-invariant and
$Z_f=B( \overline{K})+p+\Tor(X( \overline{K}))$.
\end{thm}

To prove Theorem \ref{thm_ab}, we use the following results.

\begin{lem}[cf.~{\cite[Proof of Theorem 2]{Sil17}}]\label{lem_sil}
Let $f: A \to A$ be an isogeny on an abelian variety over $ \overline{K}$.
Then there are abelian subvarieties $A_1,A_2$ of $A$ such that we have:
\begin{itemize}
\item[(1)]
The addition map $m: A_1 \times A_2 \to A$ is an isogeny.
\item[(2)]
$A_i$ is $f$-invariant for $i=1,2$.
Set $f_i=f|_{A_i}$.
\item[(3)]
The map $1_{A_1} {-} f_{1}: A_1 \to A_1$ is surjective.
\item[(4)]
$\delta_{f_2}=1$.
\end{itemize}
\end{lem}

\begin{rem}\label{split_nonisog}
Let $F \colon A \longrightarrow A$ be a surjective endomorphism.
Then we can write $F=\tau_{a} \circ f$ where $f$ is an isogeny and $\tau_{a}$ is a translation by $a\in A$.
Apply Lemma \ref{lem_sil} to $f$. Take $(a_{1}, a_{2}) \in A_{1}\times A_{2}$ such that $a_{1}+a_{2}=a$.
Then we have $F \circ m = m\circ ((\tau_{a_{1}}\circ f_{1}) \times (\tau_{a_{2}}\circ f_{2}))$,
$\tau_{a_{1}}\circ f_{1}$ is conjugate to an isogeny by a translation, and $\delta_{\tau_{a_{2}}\circ f_{2}}=1$ (cf.~\cite[\S 5]{MS20}).
\end{rem}

\begin{thm}[{\cite[Theorem 1]{KS16b}}]\label{thm_ksforab}
Let $f: A \to A$ be an isogeny on an abelian variety with $\delta_f >1$.
Take a symmetric nef $\mathbb R$-divisor $D$ such that $D$ is not numerically trivial and $f^*D \equiv \delta_f D$.
Take the canonical height function $\hat h_{D,f}$ associated to $D$:
$$\hat h_{D,f}(x)=\lim_{n \to \infty} h_D(f^n(x))/\delta_f^n.$$
Then there is an $f$-invariant proper abelian subvariety $B \subset A_{\overline K}$ such that the zero locus of $\hat h_{D,f}$ is equal to 
$B( \overline{K})+\Tor(A( \overline{K}))$.
\end{thm}

\begin{proof}[{\it Proof of Theorem \ref{thm_ab}}]

First we assume that $f$ is an isogeny, and show:

\begin{cl}\label{claim_ab}
Let $f: X \to X$ be an isogeny on an abelian variety with $\delta_f>1$.
Then there is an $f$-invariant proper abelian subvariety $B \subset X_{\overline K}$ such that
$Z_f=B( \overline{K})+\Tor(X( \overline{K}))$.
\end{cl}

\begin{proof}
We prove the claim by induction on the dimension.
If $\dim (X)=1$, then $Z_f = \Prep(f)$.
So we have to show $\Prep(f)=\Tor(X( \overline{K}))$.
Since $f$ is an isogeny, $f(X_{\overline{K}}[N])\subseteq X_{\overline{K}}[N]$ for any $N>0$, where $X_{\overline K}[N]$ denotes the (finite) set of $N$-torsion points of $X_{\overline K}$.
So $\Tor(X( \overline{K})) \subseteq \Prep(f)$.
Conversely, take any $x \in \Prep(f)$.
Then $f^{n+k}(x)=f^n(x)$ for some $n,k \in \mathbb Z_{>0}$.
This means that $x$ is contained in $(1_X-f^k)^{-1}(\Ker(f^n))$,
which is a finite group since $f^n$ and $1_X-f^k$ are isogenies.
Here we used the assumption $\delta_f>1$ to guarantee $f^k \neq 1_X$.
Therefore, $x \in \Tor(X( \overline{K}))$.
Hence $Z_f=\Tor(X( \overline{K}))$.

Assume $\dim (X) \ge 2$ and the claim holds for abelian varieties of smaller dimensions.
Take a nef symmetric $\mathbb R$-Cartier divisor $D$ on $X$ such that $D$ is not numerically trivial and $f^* D \equiv \delta_f D$.
Then Theorem \ref{thm_ksforab} implies that the zero locus of $\hat h_{D,f}$ is equal to 
$B( \overline{K})+\Tor(X( \overline{K}))$ for some proper $f$-invariant abelian subvariety $B \subset X_{\overline K}$.
For any point $x \in X_{\overline K}$, $\hat h_{D,f}(x)>0$ implies that $\alpha_f(x)=\delta_f$.
So $Z_f \subseteq B( \overline{K})+\Tor(X( \overline{K}))$.
Note that $\hat h_{D,f}$ is non-negative (cf.~\cite[Section 6]{KS16b}).

Set $f_B=f|_B: B \to B$.
Then $\delta_{f_B} \le \delta_f$ by \cite[Proposition A.11]{NZ09}.
If $\delta_{f_B}<\delta_f$, then the arithmetic degree of any point of $B$ is smaller than $\delta_f$.
Then $Z_f=B( \overline{K})+\Tor(X( \overline{K}))$.
Next, assume that $\delta_{f_B}=\delta_f$.
Then the induction hypothesis implies that there is a proper $f_B$-invariant abelian subvariety 
$C \subset B$ such that $Z_{f_B}=C( \overline{K})+\Tor(B( \overline{K}))$.
Now we have $C( \overline{K}) +\Tor(X( \overline{K}))  \subseteq Z_f$.
Conversely, take any $x \in Z_f$.
Then $mx$ is in $B$ (and also in $[m]_{X}(Z_f) = Z_f$; see Lemma \ref{lem_surj})
for some $m \in \mathbb Z_{>0}$.
Now $mx \in Z_f \cap B=Z_{f_B}$, so $kmx \in C$ for some $k \in \mathbb Z_{>0}$.
Thus $Z_f=C( \overline{K})+\Tor(X( \overline{K}))$.
This proves the claim.
\end{proof}

Next we consider a general endomorphism $f$.
Write $f=\tau_a \circ \phi$ such that $\phi$ is an isogeny and $\tau_a$ is the translation map by $a \in X$.
By Lemma \ref{lem_sil},
there are abelian subvarieties $A_1,A_2$ of $X$ such that we have:
\begin{itemize}
\item[(1)]
The addition map $m: A_1 \times A_2 \to X$ is an isogeny.
\item[(2)]
$A_i$ is $\phi$-invariant for $i=1,2$.
Set $\phi_i=\phi|_{A_i}$.
\item[(3)]
The map $1_{A_1} -\phi_{1}: A_1 \to A_1$ is surjective.
\item[(4)]
$\delta_{\phi_2}=1$.
\end{itemize}

Take $a_i \in A_i$ for $i=1,2$ such that $a=a_1+a_2$ and set $f_i=\tau_{a_i} \circ \phi_i$.
Take $p \in A_1$ satisfying $p-\phi_{1}(p)=a_1$.
Since translations act trivially on the Neron-Severi lattice and the map $\phi_1 \times \phi_2$
descends to $\phi$ via the map $m$, Lemma \ref{lem_dyn_deg_rel} and the product formula imply
$\delta_f = \delta_{\phi} = \delta_{\phi_1 \times \phi_2} = \delta_{\phi_1} = \delta_{f_1}$.
Now we have the commutative diagram:
\[
\xymatrix{
A_1 \times A_2 \ar[r]^{\phi_1 \times f_2} \ar[d]_{\tau_{p} \times 1_{A_2}} & A_1 \times A_2 \ar[d]^{\tau_{p} \times 1_{A_2}}\\
A_1 \times A_2 \ar[r]^{f_1 \times f_2} \ar[d]_{m} & A_1 \times A_2 \ar[d]^{m}\\
X \ar[r]^{f}& X
}
\]
Applying Claim \ref{claim_ab} to $\phi_1$, there is a proper $\phi_1$-invariant abelian subvariety 
$B_1 \subset (A_1)_{\overline K}$ such that $Z_{\phi_1}=B_1( \overline{K})+\Tor(A_1( \overline{K}))$.
Since $\phi_1 \times f_2$ descends to $f$ via a finite morphism,
the first equality below follows from Lemma \ref{lem_surj}, while the second follows from
Lemma \ref{lem_prod}:
\begin{align*}
Z_f &=(m \circ (\tau_p \times 1_{A_2}))(Z_{\phi_1 \times f_2}) \\
&= (m \circ (\tau_p \times 1_{A_2}))((B_1( \overline{K})+\Tor(A_1( \overline{K}))) \times A_2( \overline{K}))\\
&= B_1( \overline{K})+\Tor(A_1( \overline{K}))+p+A_2( \overline{K}).
\end{align*}
Set $B=B_1+(A_2)_{\overline K}$.
Then we have $Z_f=B( \overline{K})+p+\Tor(X( \overline{K}))$.
We compute
\begin{align*}
f(B+p) &=\phi(B_1+(A_2)_{\overline K}+p)+a \\
&=\phi(B_1)+\phi((A_2)_{\overline K})+\phi(p)+a_1+a_2 \\
&=B_1+(A_2)_{\overline K}+\phi(p)+a_1+a_2 \\
&=(B+a_2)+(\phi(p)+a_1) \\
&=B+p.
\end{align*}
So $B+p$ is $f$-invariant.
\end{proof}

\begin{thm}\label{thm_ab2}
Notation is as in Theorem \ref{thm_ab}.
Write $Z_f=B( \overline{K})+p+\Tor(X( \overline{K}))$. Then:
\begin{itemize}
\item[(1)]
$\{B+p+t \}_{t \in \Tor(X({\overline K}))}$ is the family of all maximal $f$-preperiodic subvarieties of small dynamical degree in $X_{\overline{K}}$.
\item[(2)]
There are only finitely many maximal $f$-invariant subvarieties of small dynamical degree in $X_{\overline{K}}$ (cf.~Question \ref{ques_subvar}(3)).
\item[(3)]
For any $d>0$,
$Z_f(d) \subseteq \bigcup_{i=1}^r (B( \overline{K})+p+t_i)$ for some $t_1, \ldots, t_r \in \Tor(X( \overline{K}))$.
In particular, $Z_f(d)$ is not Zariski dense.
\end{itemize}
\end{thm}

\begin{proof}
(1)
Write $f=\tau_a \circ \phi$ such that $\phi$ is an isogeny and $\tau_a$ is the translation map by $a \in X$.
Pick $t \in \Tor(X(\overline K))$. Since
$f(B + p) = B + p$, inductively, for any $s \ge 1$, we get
\begin{align*}
f^s(B+p+t) &= f^{s-1}(\phi(B+p+t)+a) \\
&=f^{s-1}(\phi(B+p)+\phi(t)+a) \\
&=f^{s-1}(f(B+p)+\phi(t)) \\
&=f^{s-1}(B+p+\phi(t)) \\
&= \cdots = B + p + \phi^s(t).
\end{align*}
Since $\phi$ is an isogeny, if $t$ is in the set of $N$-torsion points (which is a finite set),
then so is $\phi^s(t)$.
Thus $B+p+t$ is an $f$-preperiodic subvariety of small dynamical degree.
The maximality of $B+p+t$ is clear, since $Z_f=\bigcup_{t \in \Tor(X( \overline{K}))} (B( \overline{K})+p+t)$.

(2) Extending $K$, we may assume that $B$ and $p$ are defined over $K$.
Fix $d>0$.
Take the quotient $\pi: X \to Y :=X/B$.
Then $\pi$ is $f$-equivariant. Set $g :=f|_Y$.
If $B+p+t$ is $f$-invariant, then $g(\pi(p+t))=g(\pi(B+p+t))=\pi(f(B+p+t))=\pi(B+p+t)=\pi(p+t)$.
By (1), it suffices to show that there are only finitely many $g$-fixed points in $Y_{\overline K}$.
Suppose the contrary.
Then $g|_Z=\id$ for some postive dimensional irreducible closed subvariety $Z\subseteq Y_{\overline{K}}$.
By the product formula, we have $\delta_{f|_{\pi^{-1}(Z)}}=\delta_{f|_{B+p}}$.
By \cite[Theorem 1.6]{MSS18},  $\delta_{f|_{B+p}}$ is equal to the arithmetic degree of $f|_{B+p}$ at some $ \overline{K}$ point of 
$B+p$. Since $B( \overline{K}) + p \subset Z_{f}$, we have $\delta_{f|_{B+p}}<\delta_f$ and hence $\delta_{f|_{\pi^{-1}(Z)}} < \delta_{f}$.
However, $\pi^{-1}(Z)\supsetneq B+p+t$ for some $t\in \Tor(X( \overline{K}))$, a contradiction with the maximality.

(3) Fix $d>0$.
Then $\pi(Z_f(d)-p) \subseteq \Tor(Y( \overline{K})) \cap Y(d)$ and $\Tor(Y( \overline{K})) \cap Y(d)$ is a finite group.
So
$$\pi \left( [N]_X(Z_f(d)-p) \right)= [N]_Y \left( \pi(Z_f(d)-p) \right)=0$$
for some $N>0$.
Hence $Z_f(d)-p \subseteq [N]_X^{-1}(B( \overline{K}))=B( \overline{K})+ X_{\overline K}[N]$, where $X_{\overline K}[N]$ denotes the set of $N$-torsion points of $X(\overline K)$.
As a consequence, $Z_f(d)$ is contained in a finite union of subvarieties of the form $B( \overline{K})+p+t$ with $t$ being a torion point.
\end{proof}

We can use the above precise description of $Z_f$ to prove Conjecture \ref{conj_zf} for varieties of maximal Albanese dimension.

\begin{defn}\label{defn_maxalb}
Let $X$ be a normal projective variety over an algebraically closed field.
Then there is an abelian variety $\Alb(X)$, called the {\it Albanese variety} of $X$, and a dominant rational (Albanese) map $\alb_X : X \dasharrow \Alb(X)$, such that every rational map from $X$ to another abelian variety factors through $\alb_X$.

We say that  a variety $X$ is \textit{of maximal Albanese dimension} if
$\dim (\alb_X(X)) = \dim (X)$.
We say that a variety $X$ defined over $K$ is \textit{of maximal Albanese dimension} if so is
$X_{\overline{K}}$.
\end{defn}

\begin{cor}\label{cor_maxalb}
Conjecture \ref{conj_zf} holds for surjective endomorphisms on normal projective varieties of maximal Albanese dimension.
\end{cor}

\begin{proof}
Let $X$ be a normal projective variety of maximal Albanese dimension and $f: X \to X$ a surjective endomorphism with $\delta_f>1$.
By Lemma \ref{lem_rat} and replacing $X$ and $\alb_X$ by the graph, we may assume that the albanese map $\alb_X : X \to \Alb(X) = :A$ is a well-defined morphism,
and also defined over $K$ after extending $K$.
Then $f$ induces a surjective endomorphism $g: A \to A$.
By Lemma \ref{lem_surj} or
\ref{lem_rat} and
replacing $(X,f)$ by $(\alb_X(X),g|_{\alb_X(X)})$,
we may assume that $X$ is a $g$-invariant subvariety of $A$ such that $0 \in X$ and $X$ is not contained in any proper abelian subvariety of $A$.

Note that $\delta_f \le \delta_g$
(cf.~\cite[Proposition A.11]{NZ09}).
We prove the assertion by induction on $\dim (A)$.
If $\dim (A)=1$, then $X$ is either a point or $A$, so the assertion follows from Theorem  \ref{thm_ab2}.

Consider the general case. By Theorem \ref{thm_ab},
$Z_g=B( \overline{K})+p+\Tor(A( \overline{K}))$, where $B+p$ is a proper $g$-invariant translated abelian subvariety of $A_{\overline K}$.
If $\delta_f < \delta_g$, then $X( \overline{K}) \subseteq Z_g$.
This implies that $X_{\overline K} \subseteq B+p+t$ for some $t \in \Tor(A( \overline{K}))$.
Then $B+p+t=B$ since $0 \in X_{\overline K} \subseteq B+p+t$, but this contradicts that $X$ is not contained in any proper abelian subvariety of $A$.
So we have $\delta_f=\delta_g$.

Take any $d>0$.
By Theorem \ref{thm_ab2}, $Z_g(d) \subseteq \bigcup_{i=1}^r (B( \overline{K})+p+t_i)$ for some torsion points $t_1, \ldots, t_r \in A( \overline{K})$.
Then
$$Z_f(d)=X( \overline{K}) \cap Z_g(d) \subseteq X( \overline{K}) \cap \left( \bigcup_{i=1}^r (B( \overline{K})+p+t_i) \right)
=\bigcup_{i=1}^r (X( \overline{K}) \cap (B( \overline{K})+p+t_i)).$$
Moreover, $X_{\overline K} \not\subseteq B+p+t_i$ for any $i$.
So $\bigcup_{i=1}^r (X_{\overline K} \cap (B+p+t_i))$ is a proper closed subset of $X_{\overline K}$.
\end{proof}

A variety $X$ over $K$ is $Q$-abelian if so is $X_{\overline{K}}$.
We can also prove Conjecture \ref{conj_zf} for Q-abelian varieties.

\begin{cor}[cf.~{\cite[Theorem 2.8]{MZ19a}}]\label{cor_qab}
Conjecture \ref{conj_zf} holds for surjective endomorphisms on Q-abelian varieties.
\end{cor}

\begin{proof}
Let $f: X \to X$ be a surjective endomorphism on a $Q$-abelian variety.
After enlarging the ground field if necessary,
we can take a finite surjective morphism $\pi: A \to X$ from an abelian variety and a surjective endomorphism $g: A \to A$ such that $\pi \circ g=f \circ \pi$ (cf.~\cite{NZ10} or \cite[Corollary 8.2]{CMZ20}).
Then the assertion follows from Lemma \ref{lem_surj} and Theorem \ref{thm_ab2}.
\end{proof}

\section{Varieties admitting int-amplified endomorphisms}

In this section, the ground field is $ \overline{K}$ (or any algebraically closed field of characteristic zero) unless otherwise stated.
We focus on a $\Q$-factorial klt projective variety $X$ admitting an int-amplified endomorphism $\mathcal{I}$, i.e., $\mathcal{I}^*L-L=H$ for some ample Cartier divisors $L$ and $H$.
Alternatively, it is equivalent to saying that all the eigenvalues of $\mathcal{I}^*|_{\N^1(X)}$ are of modulus greater than $1$; see \cite[Theorem 1.1]{Men20}.
Given any (not necessarily int-amplified) surjective endomorphism $f$ of $X$, we may run the $f$-equivariant minimal model program (after some iteration of $f$) by \cite[Theorem 1.2]{MZ20}.
In this way, it suffices for us to focus on the following Case TIR when studying the Kawaguchi--Silverman Conjecture \ref{conj_ks}; see \cite[Theorem 1.7]{MZ19a}.

\par \vskip 1pc \noindent
{\bf Case TIR}$_n$ (Totally Invariant Ramification case).
Let $X$ be a normal projective variety of dimension $n \ge 1$, which has only $\Q$-factorial Kawamata log terminal (klt) singularities and admits an int-amplified endomorphism.
Let $f:X\to X$ be a surjective endomorphism.
Moreover, we impose the following conditions.
\begin{itemize}
\item[(A1)]
The anti-Iitaka dimension $\kappa(X,-K_X)=0$; $-K_X$ is nef, whose class is extremal in both the {\it nef cone} $\Nef(X)$ and the {\it pseudo-effective divisors cone} $\PE^1(X)$.
\item[(A2)]
$f^*D = \delta_f D$ for some prime divisor $D\sim_{\Q} -K_X$.
\item[(A3)]
The ramification divisor of $f$ satisfies $\Supp R_f = D$.
\item[(A4)]
There is an $f$-equivariant Fano contraction $\pi:X\to Y$ with $\delta_f>\delta_{f|_Y}$ ($\ge 1$).
\item[(A5)]
$\dim(X) \ge \dim(Y)+2 \ge 3$.
\end{itemize}

For the sAND Conjecture \ref{conj_zf}, we show that the same strategy works.

\begin{thm}\label{thm_int-amp-TIR}(cf.~\cite[Theorem 1.7]{MZ19a})
Let $X$ be a normal projective variety over $K$ such that $X_{ \overline{K}}$ has only $\Q$-factorial Kawamata log terminal (klt)  singularities and admits an int-amplified endomorphism.
Then we have:
\begin{itemize}
\item[(1)] If $K_{X_{ \overline{K}}}$ is pseudo-effective, then Conjecture \ref{conj_zf} holds for any surjective endomorphism of $X$.
\item[(2)] Suppose that Conjecture \ref{conj_zf} holds for Case TIR
(for those $f|_{X_i} : X_i \to X_i$ appearing in any equivariant MMP starting from $X_{ \overline{K}}$).
Then Conjecture \ref{conj_zf} holds for any surjective endomorphism on $X$.
\end{itemize}
\end{thm}

\begin{proof}
We apply almost the  same proof of \cite[Theorem 1.7]{MZ19a}.
Indeed, we only need to replace \cite[Lemma 2.5 and Theorem 2.8]{MZ19a} for Conjecture \ref{conj_ks} by our Lemma \ref{lem_surj} and Corollary \ref{cor_qab} for Conjecture \ref{conj_zf}.
\end{proof}

In the rest of this section, we show that Case TIR will not occur in many cases.

We first extend \cite[Theorem 4.4]{MY19} a bit to make all surjective endomorphisms liftable (after suitable iteration).

\begin{thm}\label{thm-MY}(cf.~\cite[Theorem 4.4]{MY19})
In Case TIR, there exists the following commutative diagram:
$$\xymatrix{
X \ar[d]_{\pi}& \ar[l]_{\mu_{X}} \widetilde{X} \ar[d]^{ \widetilde{\pi}}\\
Y & \ar[l]^{\mu_{Y}} A
}$$
where $A$ is an abelian variety, $\mu_Y$ is finite surjective, $\widetilde{X}$ is the normalization of the main component of $X\times_Y A$, and $\mu_X$ is quasi-\'etale finite surjective, such that any commutative diagram of surjective morphisms
$$\xymatrix{
h_X\acts X\ar[r]^{\pi}  & Y \racts h_Y
}$$
can be lifted equivariantly to

$$\xymatrix{
h_X \acts X \ar@<2.3ex>[d]_{\pi}& \ar[l]_{\mu_{X}} \widetilde{X} \ar@<-2.3ex>[d]^{ \widetilde{\pi}} \racts h_{\widetilde{X}} \\
h_Y \acts Y & \ar[l]^{\mu_{Y}} A \racts h_{A}
}$$
\end{thm}

\begin{proof}
Note that in Case TIR, $\kappa(X,-K_X)=0$ and $-K_X\sim_{\Q} D$ for some reduced effective divisor $D$.
By \cite[Theorem 6.2]{MZ19a}, $\Supp R_{h_X}\subseteq D$ for any surjective endomorphism $h_X$ of $X$.
Next, we only consider those $h_X$ which can descend equivariantly to $h_Y$ via $\pi$.
Since $D$ is $\pi$-ample, $\pi(D)= Y$.
Note also that $D$ is irreducible in our assumption.

Using the same notation in the proof of \cite[Theorem 4.4]{MY19}, we can define a positive integer $m_E$ for any prime divisor $E$ on $Y$ so that $\pi^*E=m_E F$ for some prime divisor $F$ on $X$.
We set $\Delta=\sum \frac{m_E-1}{m_E}E$ which is a finite sum.
By the same computation in the proof of  \cite[Theorem 4.4]{MY19} and the key reason $\pi(D)=Y$, we have $h_Y^*(K_Y+\Delta)\sim K_Y+\Delta$.
Since the int-amplifed endomorphism on $X$ descends to an int-amplified endomorphism on $Y$,
$K_Y+\Delta\sim_{\Q} 0$ by \cite[Theorem 1.1]{Men20} and \cite[Proposition 3.4]{MZ19a}.

Applying \cite[Lemma 4.9]{MY19} to the pair $(Y,\Delta)$, we obtain the commutative diagram:
$$\xymatrix{
Y_1 \ar[d]_{\mu_{Y_1}}\ar[r]^{h_{Y_1}}&  Y_1 \ar[d]^{\mu_{Y_1}}\\
Y \ar[r]^{h_Y} &  Y
}$$
where $K_{Y_1}\sim 0$ and $\mu_{Y_1}$ is finite surjective which does not depend on $h_Y$.
In the proof of \cite[Theorem 4.4]{MY19}, it is further shown that $Y_1$ is klt.
Note that $Y_1$ admits an int-amplified endomorphism.
So $Y_1$ is $Q$-abelian by \cite[Theorem 1.9]{Men20}.
Further, we can find abelian variety $A$ and the commutative diagram:
$$\xymatrix{
A \ar[d]_{\mu_{Y_2}}\ar[r]^{h_A}&  A \ar[d]^{\mu_{Y_2}}\\
Y_1 \ar[r]^{h_{Y_1}} &  Y_1
}$$
where $\mu_{Y_2}$ is quasi-\'etale, finite and surjective (cf.~\cite[Corollary 8.2]{CMZ20}).

Now $h_A$ is the lifting of $h_Y$ via
$\mu_Y :=\mu_{Y_1}\circ\mu_{Y_2}$.
The induced cover $\mu_X$ is quasi-\'etale by \cite[Lemma 4.5]{MY19}.
\end{proof}

Let $X$ be a normal projective variety.
Denote by $\pi_1^{\alg}(X_{\reg})$ the {\it algebraic fundamental group} of the smooth locus $X_{\reg}$ of $X$.
Conjecture \ref{conj_ks} for smooth rationally connected projective varieties admitting int-amplified endomorphisms has been solved in  \cite{MZ19a} and \cite{MY19}. Indeed, the key point of the proof is the following, noting that a smooth rationally connected projective variety has trivial algebraic fundamental group.

\begin{thm}\label{thm_iamp-src}(cf.~\cite{MZ19a}, \cite{MY19})
Let $X$ be a $\mathbb Q$-factorial klt projective variety with the group $\pi_1^{\alg}(X_{\reg})$ being finite.
Suppose $X$ admits an int-amplified endomorphism.
Then Case TIR will not occur during any MMP starting from $X$.
\end{thm}

\begin{proof}
We argue along the lines of \cite[Section 4.3]{MY19} (while \cite{MZ19a} is effective in low dimension or relative dimension).
Suppose the contrary that there exists some $f$-equivariant (after iteration) MMP
$$X_1\dashrightarrow\cdots \dashrightarrow X_r\to Y,$$
such that $(X_r, Y, f|_{X_r}, f|_Y)$ satisfies Case TIR.
Let $\mathcal{I}:X\to X$ be an int-amplified endomorphism such that the MMP is also $\mathcal{I}$-equivariant.
Note that $(X_r, Y, \mathcal{I}|_{X_r}, \mathcal{I}|_Y)$ may not satisfy Case TIR.
Nevertheless, by Theorem \ref{thm-MY}, $\mathcal{I}|_{X_r}$ satisfies $(*)$ in \cite[Definition 4.11]{MY19}.
By \cite[Lemmas 4.12 and 4.13]{MY19}, $\mathcal{I}$ also satisfies $(*)$ in \cite[Definition 4.11]{MY19}.
In particular, without considering any dynamic property, we have the commutative diagram:
$$\xymatrix{
X \ar[d]_{\pi}& \ar[l]_{\mu_{X}} \widetilde{X} \ar[d]^{ \widetilde{\pi}} \\
Y & \ar[l]^{\mu_{Y}} A
}$$
where $A$ is an abelian variety, $\mu_Y$ is finite surjective, and $\mu_X$ is quasi-\'etale finite surjective.
Note that $\dim(A)=\dim(Y)>0$.
Then we can take the multiplication map $[n]_A:A\to A$ with $n$ allowed to be arbitrarily large.
Taking the base change, we see that the natural projection $\widetilde{X}\times_{A,[n]_A} A\to \widetilde{X}$ is an \'etale cover of arbitrarily large degree.
In particular, $\pi_1^{\alg}(X_{\reg})$ is infinite, which is a contradiction.
\end{proof}

In \cite[Theorem 8.6]{MZ19a}, the relative dimensional one case has been ruled out in Case TIR.
Based on this, we will deal with the relative dimensional two case.

We will show that Case TIR is equivalent to the following simpler Case TIR'.

\par \vskip 1pc \noindent
{\bf Case TIR'}$_n$ (Totally Invariant Ramification case).
Let $X$ be a normal projective variety of dimension $n \ge 1$, which has only $\Q$-factorial Kawamata log terminal (klt) singularities and admits an int-amplified endomorphism.
Let $f:X\to X$ be a surjective endomorphism.
Moreover, we impose the following conditions.
\begin{itemize}
\item[(A1')]
The anti-Iitaka dimension $\kappa(X,-K_X)=0$.
\item[(A2')]
$f^*D = \delta_f D$ for some reduced effective Weil divisor $D\sim_{\Q} -K_X$.
\item[(A4)]
There is an $f$-equivariant Fano contraction $\pi:X\to Y$ with $\delta_f>\delta_{f|_Y}$ ($\ge 1$).
\end{itemize}

\begin{lem}\label{lem-tir-equi}
Case TIR$_n$ is equivalent to Case TIR'$_n$.
\end{lem}

\begin{proof}
Assume that Case TIR'$_n$ holds.
We first show A2 that $D$ is irreducible.
Let $P$ and $Q$ be two different irreducible components of $D$.
If $P$ is not $\pi$-ample, then $P\equiv \pi^*P_Y$ for some $P_Y\in \N^1(Y)\backslash 0$ by the cone theorem (cf.~\cite[Theorem 3.7]{KM98}).
Note that $g^*P_Y=\delta_f P_Y$.
This is a contradiction with the assumption A4.
So we may assume $P$ and $Q$ are $\pi$-ample.
Then $P-tQ$ is $\pi$-trivial for some rational number $t>0$ and the same argument showes that $P-tQ\equiv 0$.
Note that $f^*(P-tQ)=\delta_f(P-tQ)$ and the Albanese morphism of $X$ is surjective by \cite[Theorem 1.8]{Men20}.
Then $P-tQ\sim_{\Q} 0$ by \cite[Proposition 3.3]{MZ19a}.
In particular, $\kappa(X, D)\ge \kappa(X, P)>0$, a contradiction with A1'.

Note that $\N^1(X)/\pi^*\N^1(Y)$ is 1-dimensional and $f^*|_{\N^1(X)/\pi^*\N^1(Y)} = q \id$ for some integer $q>0$.
	Then $q = \delta_f$ is the only eigenvalue of $f^*|_{\N^1(X)}$ with modulus $\delta_f$ and the $q$-eigenspace is $1$-dimensional.
	On the other hand $f^*H\equiv \delta_f H$ for some nef divisor $H$ by a version of the Perron-Frobenius theorem (cf.~\cite{Birk}).
	So the rays $R_H=R_D=R_{-K_X}$ in
	$\Nef(X)$ and $\PE^1(X)$,
and the second part of A1 holds, by \cite[Lemma 9.1]{MZ19a}.

By \cite[Theorem 6.2]{MZ19a}, $\Supp R_f\subseteq D$. On the other hand, A2' and $\delta_f>1$ imply A3 that $\Supp R_f=D$.
A5 follows from \cite[Theorem 8.6]{MZ19a}.
\end{proof}

\begin{prop}\label{prop-ti-kappa>0}
Let $\pi:X\to Y$ be the Fano contraction of a $K_X$-negative extremal ray $R_C$ on a $\Q$-factorial klt projective variety $X$ to an abelian variety $Y$ with $\dim(X)=\dim(Y)+1$.
Let $f:X\to X$ and $g:Y\to Y$ be int-amplified endomorphisms such that $\pi\circ f=g\circ \pi$ and $\delta_f>\delta_g$.
Then we have:
\begin{itemize}
\item[(1)]
$R_C$ is the only $K_X$-negative extremal ray.
\item[(2)]
The anti-Iitaka dimension $\kappa(X,-K_X)>0$, and $f^*K_X\equiv \delta_f K_X$.
\item[(3)] There is an $f$-equivariant surjective morphism $\tau:X\to Z\cong\mathbb{P}^1$ such that $f|_Z$ is $\delta_f$-polarized.
\item[(4)] $P$ is semi-ample for any $f^{-1}$-invariant prime divisor $P$.
\end{itemize}
\end{prop}

\begin{proof}
(1) Let $R_{C_2}$ be a $K_X$-negative extremal ray and $\pi_2:X\to X_2$ the contraction of $R_{C_2}$.
By \cite[Theorem 3.7]{KM98}, we may take $C_2$ as a rational curve.
Then $\pi(C_2)$ is a point in the abelian variety $Y$.
Let $F$ be a (connected) fibre of $\pi_2$.
Note that $\pi(C_2')$ is a point for any curve $C_2'\subseteq F$ since $C_2'\in R_{C_2}$.
Hence, $\pi(F)$ is a point.
Then the rigidity lemma \cite[Lemma 1.15]{Deb01} implies that $\pi$ factors through $\pi_2$.
Since $\pi$ itself is a contraction of $R_C$, $R_{C_2}=R_C$ and $X_2=Y$. This proves (1).

(2) By a version of the Perron-Frobenius theorem in \cite{Birk}, there exists a nef eigenvector $D$ of $f^*|_{\N^1(X)}$ with eigenvalue $\delta_f$.
Note that the eigenvalues of $f^*|_{\N^1(X)}$ consist of one $\delta_f$ and
eigenvalues of $g^*|_{\N^1(Y)}$ (with modulus $<\delta_f$).
So $D$ is $\pi$-ample and extremal in $\PE^1(X)$ by \cite[Lemma 9.1]{MZ19a}.

Let $a > 0$ be the rational number such that $B:=aK_X+D$ is $\pi$-trivial, i.e., $B\cdot C=0$.
We claim that $B$ is pseudo-effective.
Suppose the contrary.
For a small effective ample $\Q$-Cartier divisor $E$, $(1/a)B + E$ is not pseudo-effective.
	Denote by $A := E + (1/a)D$ which is ample since $D$ is nef.
	Thus $K_X+A = (1/a)B + E$ is not pseudo-effective.
	Note that $(X, A)$ is a klt pair by a suitable choice of $A$.
	By the cone theorem (cf.~\cite[Theorem 3.7]{KM98}), there exits some $(K_X+A)$-negative extremal ray $R_{C'}$.
	Note that $(K_X+A)\cdot C=E\cdot C>0$.
	So $R_{C'}\neq R_C$, contradicting (1).
	This proves the claim.

Since $-K_X$ is effective by \cite[Theorem 1.5]{MZ19a}, the claim above and $D=B+(-aK_X)$ imply $R_D=R_{-K_X}$ in $\PE^1(X)$  and hence $B \equiv 0$ and $D\equiv -aK_X$.
We assert that $\kappa(X, -K_X)>0$.	
Suppose the contrary.
Then $-K_X\sim_{\Q} D'$ for some reduced Weil divisor $D'$ such that $f^{-1}(D')=D'$ and $\Supp R_f=D'$ by \cite[Theorem 1.5]{MZ19a}.
Note that $f^*D'\equiv f^*D/a=(\delta_f/a) D\equiv \delta_f D'$ and $\delta_f D'-f^*D'\ge 0$.
Then $f^*D'=\delta_f D'$.
In particular, Case TIR' (and hence Case TIR) is satisfied.
However, this contradicts the condition A5 of Case TIR.
So $\kappa(X, -K_X)>0$ as asserted.
This also proves (2).

(3) By \cite[Theorem 7.8]{MZ19a}, there exists an $f$-equivariant dominant rational map $\tau:X\dashrightarrow Z$ to a normal projective variety $Z$ (with connected general fibres) such that $\dim(Z)>0$ and $f|_Z$ is $\delta_f$-polarized.
Note that the indeterminacy locus of $\tau$ is an $f^{-1}$-invariant closed subset of codimension $\ge 2$ in $X$.
By \cite[Lemma 7.5]{CMZ20} and \cite[Lemma 4.1]{Men20}, $\tau$ is then a morphism since $\dim(A)=\dim(X)-1$.

We claim that $Z \cong \PP^1$.
Let $h:=f|_Z$.
Write $h^*H_Z\sim \delta_f H_Z$ for some very ample divisor $H_Z$.
Let $H_X:=\tau^*H_Z$.
Then $f^*H_X\sim \delta_f H_X$.
So $H_X\equiv bD$ for some $b>0$.
By \cite[Proposition 8.4]{MZ19a}, $H_X^2\equiv_w b^2D^2\equiv_w 0$.
This is possible only when $\dim(Z)=1$.
Since $H_X\equiv bD$ is $\pi$-ample,
the general fibre $F$ of $\pi$ (which is a smooth rational curve) is not contracted by $\tau$
and hence dominates $Z$.
So the claim and hence (3) are proved.

(4) Let $P$ be an $f^{-1}$-invariant prime divisor.
By \cite[Lemma 7.5]{CMZ20} and \cite[Lemma 4.1]{Men20}, $P$ dominates $Y$.
Then $f^*P=\delta_f P$ and $R_P=R_D$ in $\PE^1(X)$.
Take $H_Z=z$ for some general point of $Z$ and let $X_z$ be the fibre of $\tau$.
Note that $X_z\cdot P\equiv_w 0$.
So either $X_z\cap P=\emptyset$ or $X_z\cap P=P$.
In particular, this implies that $\tau(P)$ is a point.
We claim that $P=\tau^{-1}(\tau(P))$.
If the claim is false, then $\tau^{-1}(\tau(P))$ is reducible; hence there exists some irreducible component $Q\neq P$ of $\tau^{-1}(\tau(P))$ such that $P\cap Q\neq \emptyset$, since $\tau$ has connected fibres.
Note that $\tau(P)$ is $h^{-1}$-periodic by \cite[Lemma 7.5]{CMZ20}.
So $Q$ and hence $P\cap Q$ are $f^{-1}$-periodic.
By \cite[Lemma 7.5]{CMZ20}, $\pi(P\cap Q)$ is $g^{-1}$-periodic and hence equal to $Y$ by \cite[Lemma 4.1]{Men20}, contradicting that $\dim(P\cap Q)<\dim(Y)$. So the claim is true.
In particular, $P$ is semi-ample and (4) is proved.
\end{proof}

\begin{thm}\label{thm-tir-nn2}
Suppose that $X$ has only terminal singularities in Case TIR.
Then $\dim(X)$ $\ge \dim(Y)+3\ge 4$.
\end{thm}

\begin{proof}
The proof consists of Claims \ref{tir3-cl-zti} to \ref{tir3-cl-=1}.
We prove by contrapositive, so suppose the contrary that $\dim(X)=\dim(Y)+2$.
Let $\mathcal{I}:X\to X$ be an int-amplified endomorphism such that $\pi : X \to Y$ is $\mathcal{I}$-equivariant.

Since $X$ has terminal singularities, $\Sing(X)$ has dimension $\le \dim(X)-3<\dim(Y)$, so it does not dominate $Y$.
By \cite[Theorem 5.1]{Fak03}, there is a Zariski dense subset $Y_0$ of $Y$ such that the fibre $X_y=\pi^{-1}(y)$ is smooth and $\mathcal{I}$-periodic for any $y\in Y_0$.
Note that for each $y\in Y_0$, $X_y$ is $\mathcal{I}^s$-invariant for some $s>0$;
and $\mathcal{I}^s|_{X\backslash D}$ is quasi-\'etale.
Shrinking $Y_0$ a bit, we may assume that $\mathcal{I}^s|_{X_y}$ is \'etale outside $D\cap X_y$ by the purity of branch loci.

Let $y\in Y_0$ be a general point.
Denote by $$m:= \text{ the number of irreducible components (curves) of } D\cap X_y.$$
Since $\pi$ is a Fano contraction, $X_y$ is a smooth Fano surface and $\mathcal{I}^s|_{X_y}$ is polarized.
By \cite[Corollary 1.4]{MZ19}, $(X_y, D\cap X_y)$ is then a toric pair.
Moreover, by the classification of smooth toric Fano surfaces, the boundary $D\cap X_y$ is a simple normal crossing loop of $m$ smooth rational curves $\ell_i$ with $m\ge 3$.
Note that $\ell_i$ and $\ell_i\cap \ell_j$ are $(\mathcal{I}^s|_{X_y})^{-1}$-periodic.
Since $\mathcal{I}^s|_{X_y\backslash D\cap X_y}$ is \'etale, any $(\mathcal{I}^s|_{X_y})^{-1}$-periodic point $x$ lies in $D\cap X_y$ by noting that $\mathcal{I}^s|_{X_y}$ is ramified at $x$.
Since each $\ell_i\cong \mathbb{P}^1$ has (exactly) two $(\mathcal{I}^s|_{X_y})^{-1}$-periodic points which are the intersecting points with some other two $\ell_j$ and $\ell_k$, $x$ further lies in $\Sing(D\cap X_y)$ by the same argument on $\mathcal{I}^{sm!}|_{\ell_i}$.
In particular, $\Sing(D\cap X_y)$ is exactly the set of $(\mathcal{I}^s|_{X_y})^{-1}$-periodic points in $X_y$.
Note that $\sharp \Sing(D\cap X_y)=m\ge 3$.

\begin{cl}\label{tir3-cl-zti}
Let $Z$ be the union of irreducible components of dimension $\dim(X)-2$ in $\Sing(D)$.
Then $h^{-1}(Z)=Z$ for any surjective endomorphism $h$ of $X$.
\end{cl}
Let $\nu:\overline{D}\to D\subseteq X$ be the normalization of $D$ and $\mathfrak{c}$ the conductor of $D$, regarded as a Weil divisor on $\overline{D}$.
	Since $X$ is smooth in codimension $2$, the adjunction formula gives
	$$K_{\overline{D}}+\mathfrak{c}=\nu^*(K_X+D)$$
	where $\nu^*(K_X+D)$ is regarded as the pullback of a divisorial sheaf.
	There is an endomorphism $\overline{h}:\overline{D}\to \overline{D}$ such that $\nu\circ \overline{h}=h\circ \nu$ and its ramification divisor $R_{\overline{h}}$ is $\overline{h}^*\mathfrak{c}-\mathfrak{c}$. In fact, $K_X+D=h^*(K_X+D)$ implies $K_{\overline{D}}+\mathfrak{c}=\overline{h}^*(K_{\overline{D}}+\mathfrak{c})$.
By choosing  an int-amplified $h=\mathcal{I}$,  $\mathfrak{c}$ is reduced and $\overline{h}^{-1}$-invariant for any $h$ (cf.~\cite[Theorem 3.3]{Men20}, \cite[Lemma 5.3, the arxiv version]{NZ10}).
Then $\mathfrak{c}=\nu^{-1}(Z)$ and $h^{-1}(Z)=Z$.
So the claim is proved.

\begin{setup}\label{n:6.1} {\bf Some notation.}
By Theorem \ref{thm-MY}, we have the commutative diagrams:
$$\xymatrix{
f \acts X \ar@<2.3ex>[d]_{\pi}& \ar[l]_{\mu_{X}} \widetilde{X} \ar@<-2.3ex>[d]^{ \widetilde{\pi}} \racts \widetilde{f} \\
g \acts Y & \ar[l]^{\mu_{Y}} A \racts g_{A}
}$$

$$\xymatrix{
\mathcal{I} \acts X \ar@<2.3ex>[d]_{\pi}& \ar[l]_{\mu_{X}} \widetilde{X} \ar@<-2.3ex>[d]^{ \widetilde{\pi}} \racts \widetilde{\mathcal{I}} \\
\mathcal{I}|_Y \acts Y & \ar[l]^{\mu_{Y}} A \racts \mathcal{I}_{A}
}$$
where $\widetilde{X}$ is a $\Q$-Gorenstein klt normal projective variety, $A$ is an abelian variety, $\mu_Y$ is a finite surjective morphism, and $\mu_X$ is a finite surjective quasi-\'etale morphism.
Note that $\mathcal{I}|_Y, \mathcal{I}_A, \widetilde{\mathcal{I}}$ are still int-amplified; see \cite[Section 3]{Men20}.
Note that $\widetilde{X}$ still has terminal singularities by \cite[Proposition 5.20]{KM98}.
Then $\widetilde{X}$ is $\Q$-factorial by Proposition \ref{prop-nqf}.
\end{setup}

\begin{cl}\label{tir3-cl->=3}
$\deg \pi|_Z =m \ge 3$.
\end{cl}

By the generic smoothness, $X_y\cap Z=X_y\cap \Sing(D)=\Sing(D\cap X_y)$ for general $y\in Y_0$.
Then $\sharp Z\cap X_y=m\ge 3$ for general $y\in Y_0$.
So the claim is proved.

\begin{cl}\label{cl-C-fs}
$\pi|_C:C\to Y$ is a finite surjective morphism for any $\mathcal{I}^{-1}$-periodic irreducible closed subvariety $C$ of codimension $2$ in $X$.
\end{cl}
Since $\pi$ has connected fibres, $\pi(C)$ is $\mathcal{I}|_Y^{-1}$-periodic by \cite[Lemma 7.5]{CMZ20} and hence $\mu_Y^{-1}(\pi(C))$ is $\mathcal{I}_A^{-1}$-periodic.
By \cite[Lemma 4.1]{Men20}, $\mu_Y^{-1}(\pi(C))=A$ and hence $\pi(C)=Y$.
Note that $\dim (C)=\dim (Y)$ and $\pi|_C$ is equidimensional (cf.~\cite[Lemma 4.3]{Men20}).
So $\pi|_C$ is finite surjective and the claim is proved.

\begin{cl}\label{cl-Z=Iti}
$Z$ is the union of $\mathcal{I}^{-1}$-periodic irreducible closed subvarieties of codimension $2$ in $X$.
\end{cl}
First note that $Z$ is $\mathcal{I}^{-1}$-invariant by Claim \ref{tir3-cl-zti}.
Let $C$ be an $\mathcal{I}^{-1}$-periodic irreducible closed subvariety of codimension $2$ in $X$.
For any $y\in Y_0$, $X_y$ is $\mathcal{I}^s$-invariant for some $s>0$.
Then $C\cap X_y$ is $(\mathcal{I}^s|_{X_y})^{-1}$-invariant.
Recall that $Z\cap X_y$ is exactly the set of $(\mathcal{I}^s|_{X_y})^{-1}$-periodic points.
By Claim \ref{cl-C-fs}, $\pi_C$ is finite surjective and hence $C\cap X_y\subseteq Z\cap X_y$.
Then $C=\overline{\bigcup_{y\in Y_0} C\cap X_y}\subseteq Z$.
So the claim is proved.

\begin{cl}\label{tir3-cl-<=2}
$Z$ has exactly one or two irreducible components.
\end{cl}
Note that the general fibre of $\pi|_D:D\to Y$ has $m\ge 3$ irreducible components.
Recall that $\nu:\overline{D}\to D$ is the normalization.
Note that the general fibre $\pi|_D\circ \nu$ has $m$ connected (and irreducible) components.

Let $B$ be an irreducible component of $\overline{D}_y:=(\pi|_D\circ \nu)^{-1}(y)$ for some general $y\in Y_0$.
Since $\overline{D}$ is normal, $\Sing(\overline{D})$ does not dominate $Y$ and hence $\overline{D}_y$ and $B$ are smooth.
Note that $\nu|_{B}: B\to \nu(B)$ is a finite surjective birational morphism since $D$ is smooth at $\nu(B)\backslash \Sing(D)\neq \emptyset$.
Recall that $\nu(B)$ is a smooth rational curve lying on the boundary divisor of the smooth toric Fano surface $X_y$.
So $\nu|_{B}$ is isomorphic.
In particular, $$\nu^{-1}(Z)\cap \overline{D}_y=\bigsqcup_{B\subseteq \overline{D}_y} (\nu^{-1}(Z)\cap B)=\bigsqcup_{B\subseteq \overline{D}_y} (\nu^{-1}(Z\cap \nu(B))\cap B).$$
Note that $\sharp Z\cap \nu(B)=2$.
So $\sharp \nu^{-1}(Z)\cap \overline{D}_y= 2m$ and hence $\deg (\pi|_D\circ \nu)|_{\nu^{-1}(Z)}=2m$.

Let $p_1:\overline{D}\to \overline{Y}$ and $p_2:\overline{Y}\to Y$ be the Stein factorization of $\pi|_D\circ \nu$.
Note that $\deg p_2=m$.
Then $\deg (p_1)|_{\nu^{-1}(Z)}=2m/m=2$.
Since $\overline{Y}$ is irreducible, $\nu^{-1}(Z)$ has at most two irreducible components.
Since $Z\neq \emptyset$, the claim is proved.

\begin{cl}\label{cl-Z-}
$\widetilde{Z}:=\mu_X^{-1}(Z)$ is the union of $\widetilde{\mathcal{I}}^{-1}$-periodic irreducible closed subvarieties of codimension $2$ in $\widetilde{X}$.
\end{cl}
Let $y\in Y_0$ be a general point such that $\mu_X$ is quasi-\'etale over $X_y$.
Since $X_y$ is a smooth toric Fano surface, $X_y$ has no non-isomorphic quasi-\'etale cover.
Then $\widetilde{X}_y:=\mu_X^{-1}(X_y)$ is a disjoint union of smooth toric Fano surfaces mapping isomorphically to $X_y$ via $\mu_X$.
Assume $X_y$ is $\mathcal{I}^s$-invariant for some $s>0$.
Then $\mu_X^{-1}(Z\cap X_y)$ is exactly the set of $(\widetilde{\mathcal{I}}^s|_{\widetilde{X}_y})^{-1}$-periodic points in $\widetilde{X}_y$.
So the claim follows from the same proof of Claim \ref{cl-Z=Iti}.

\begin{cl}\label{cl-rep}
By choosing $\mu_Y$ and $\mu_X$ suitably in
Theorem \ref{thm-MY}, we may assume further $\deg \widetilde{\pi}|_{\widetilde{C}}=1$ for any $\widetilde{\mathcal{I}}^{-1}$-periodic irreducible closed subvariety $\widetilde{C}$ of codimension $2$ in $\widetilde{X}$.
\end{cl}

Let $\widetilde{C}$ be an irreducible component of $\widetilde{Z}$.
As reasoned in Claim \ref{cl-C-fs}, $\widetilde{\pi}|_{\widetilde{C}}: \widetilde{C}\to A$ is finite surjective.
Let $\overline{\widetilde{C}}$ be the normalization of $\widetilde{C}$.
Note that $\overline{\widetilde{C}}$ admits an int-amplified endomorphism.
By the ramification divisor formula and \cite[Theorem 1.5]{Men20}, the induced finite surjective morphism $\overline{\widetilde{C}}\to A$ is quasi-\'etale and hence \'etale.
Then $\overline{\widetilde{C}}$ is an abelian variety by \cite[Section 18, Theorem]{Mum74}.
If $\deg \widetilde{\pi}|_{\widetilde{C}}>1$, then we replace $\widetilde{X}$ by $\widetilde{X}\times_A \overline{\widetilde{C}}$ and $A$ by $\overline{\widetilde{C}}$.
Note that $\widetilde{f}$ is also liftable after some iteration.
We show that such replacements can be done in only finitely many steps.
After each replacement, $\deg \widetilde{\pi}|_{\widetilde{Z}}$ remains unchanged, while the number of irreducible components of $\widetilde{Z}$ strictly increases.
This is because the natural embedding $\widetilde{C}\hookrightarrow \widetilde{C}\times_A \widetilde{C}$ and $\deg \widetilde{\pi}|_{\widetilde{C}}>1$ imply that the preimage of $\widetilde{C}$ in $\widetilde{X}\times_A \overline{\widetilde{C}}$ splits into at least two irreducible components.
Finally, Claim \ref{cl-Z-} implies that after each replacement, $\widetilde{Z}$ is still the union of $\widetilde{\mathcal{I}}^{-1}$-periodic irreducible closed subvariety $\widetilde{C}$ of codimension $2$ in $\widetilde{X}$.
In particular, $\widetilde{Z}$ has at most (fixed) $\deg \widetilde{\pi}|_{\widetilde{Z}}$ irreducible components.
So we can only do at most finitely many replacements till the claim is satisfied.

\par \vskip 0.5pc
In the following, we choose the $\widetilde{X}$ in Claim \ref{cl-rep}.
Denote by $\widetilde{D}:=\mu_X^{-1}(D)$.
Since $\mu_X:\widetilde{X}\to X$ is quasi-\'etale, we have $K_{\widetilde{X}}=\mu_X^*K_X\sim_{\Q} -\mu_X^*D=-\widetilde{D}$.
Since $D$ is nef, so is $\widetilde{D}$.
By \cite[Theorem 5.13]{Uen75}, $\kappa(\widetilde{X},\widetilde{D})=\kappa(X,D)=0$.
Note that $K_{\widetilde{X}}$ is not pseudo-effective.
By \cite[Theorem 1.2]{MZ20}, replacing $\widetilde{f}$ and $\widetilde{\mathcal{I}}$ by a positive power, there is an $(\widetilde{f},\widetilde{\mathcal{I}})$-equivariant birational MMP $\varphi:\widetilde{X}\dashrightarrow \widetilde{X}_r$ and a Fano contraction $\widetilde{\pi}_r:\widetilde{X}_r\to \widetilde{Y}$.
Denote by $\widetilde{f}_r:=\widetilde{f}|_{\widetilde{X}_r}$, $\widetilde{g}:=\widetilde{f}|_{\widetilde{Y}}$ and $\widetilde{\mathcal{I}}_r:=\widetilde{\mathcal{I}}|_{\widetilde{X}_r}$.
Since $A$ is an abelian variety, this MMP is over $A$; see the first paragraph of the proof of Proposition \ref{prop-ti-kappa>0}.

\begin{cl}\label{cl-tqf}
$\varphi:\widetilde{X}\dashrightarrow \widetilde{X}_r$ is a composition of divisorial contractions;
and $\widetilde{X}_r$ has $\Q$-factorial terminal singularities.
\end{cl}

Note that $K_{\widetilde{X}}=\mu_X^*K_X$.
Recall that $\widetilde{X}$ has only $\Q$-factorial terminal singularities
(cf.~\ref{n:6.1}).
Since the MMP is over $A$, whenever a flipping contraction occurs, one can find some $(\widetilde{\mathcal{I}}|_{\widetilde{Y}})^{-1}$-invariant subvariety of dimension at most $\dim(\widetilde{X})-3$ by \cite[Lemma 7.5]{CMZ20}.
Note that $\widetilde{\mathcal{I}}|_{\widetilde{Y}}$ is int-amplified and $\dim(A)>\dim(\widetilde{X})-3$.
Then we get a contradiction by \cite[Lemma 4.1]{Men20}.
The second assertion of the claim follows from \cite[Proposition 3.36, Corollary 3.43]{KM98}.
The claim is proved.

\begin{cl}\label{cl-Dr}
$\widetilde{D}_r:=\varphi(\widetilde{D})$ is a (nonzero) reduced divisor containing all the $(\widetilde{\mathcal{I}}_r)^{-1}$-periodic irreducible closed subvarieties of codimension $2$ in $\widetilde{X}_r$. Further,
$\varphi^{-1}(\widetilde{D}_r)=\widetilde{D}$,
$-K_{\widetilde{X}_r}\sim_{\Q}\widetilde{D}_r$,
$\kappa(\widetilde{X}_r, \widetilde{D}_r)=0$, and
$\widetilde{f}_r^*\widetilde{D}_r=\delta_{\widetilde{f}_r} \widetilde{D}_r$
.
\end{cl}

Note that $\widetilde{\mathcal{I}}|_{\widetilde{X}\backslash \widetilde{D}}$ is quasi-\'etale.
So the $\varphi$-exceptional divisor (being $\widetilde{I}^{-1}$-invariant) is contained in $\widetilde{D}$.
Denote by $\widetilde{D}_r^1$ the union of divisors contained in $\widetilde{D}_r$.
Then $\widetilde{\mathcal{I}}_r|_{\widetilde{X}_r\backslash \widetilde{D}_r^1}$ is quasi-\'etale and $\Supp R_{\widetilde{\mathcal{I}}_r}=\widetilde{D}_r^1$.
Let $\widetilde{C}_r$ be some $(\widetilde{\mathcal{I}}_r)^{-1}$-invariant irreducible closed subvariety of codimension $2$ in $\widetilde{X}_r$.
By \cite[Lemma 7.5]{CMZ20} and \cite[Lemma 4.1]{Men20}, $\widetilde{C}_r$ dominates $A$ and $\deg \widetilde{\mathcal{I}}_r|_{\widetilde{C}_r}=\deg \widetilde{\mathcal{I}}_A$.
By the projection formula, $\widetilde{\mathcal{I}}_r^*\widetilde{C}_r=(\deg \widetilde{\mathcal{I}}_r/ \deg \mathcal{I}_A)\widetilde{C}_r$.
Note that $\widetilde{\mathcal{I}}_r|_F$ is polarized (after iteration) for the general $\widetilde{\mathcal{I}}_r$-periodic fibre $F$ of $\widetilde{X}_r\to A$ (cf.~\cite[Theorem 5.1]{Fak03}).
So $\deg \widetilde{\mathcal{I}}_r/ \deg \mathcal{I}_A>1$.
Since $\Sing(\widetilde{X}_r)$ is of codimension $\ge 3$ in $\widetilde{X}_r$,
$\widetilde{C}_r \subseteq \widetilde{D}_r^1$ by the purity of banch loci.
Let $E$ be any prime exceptional divisor of $\varphi$.
Then $\varphi(E)$ is $\widetilde{\mathcal{I}}_r^{-1}$-periodic and dominates $A$ by \cite[Lemma 7.5]{CMZ20} and \cite[Lemma 4.1]{Men20}.
So $\varphi(E)$ is of codimension $2$ in $\widetilde{X}_r$ and $\varphi(E)\subseteq \widetilde{D}_r^1$.
In particular, $\widetilde{D}_r=\widetilde{D}_r^1$ and $\varphi^{-1}(\widetilde{D}_r)=\widetilde{D}$.
Therefore, $\kappa(\widetilde{X}_r, \widetilde{D}_r)=\kappa(\widetilde{X}, \widetilde{D})=0$ by \cite[Theorem 5.13]{Uen75}.

By the ramification divisor formula, we have $$K_{\widetilde{X}_r}+\widetilde{D}_r=\widetilde{\mathcal{I}}_r^*(K_{\widetilde{X}_r}+\widetilde{D}_r).$$
Therefore, $-K_{\widetilde{X}_r}\sim_{\Q} \widetilde{D}_r$ by applying \cite[Theorem 1.1]{Men20} and \cite[Proposition 3.3]{MZ19a}.
Note that $\widetilde{f}^*\widetilde{D}=\delta_{\widetilde{f}} \widetilde{D}$ and hence $\widetilde{f}_r^*\widetilde{D}_r=\delta_{\widetilde{f}_r} \widetilde{D}_r$.
So the claim is proved.

\begin{cl}\label{tir3-cl-xrtir}
$(\widetilde{\pi}_r: \widetilde{X}_r\to \widetilde{Y}, \widetilde{f}_r, \widetilde{g})$ satisfies Case TIR with $\dim(\widetilde{X}_r)=\dim(\widetilde{Y})+2$ and $\widetilde{Y}=A$.
\end{cl}

If $\delta_{\widetilde{f}_r}>\delta_{\widetilde{g}}$, then Case TIR' and hence Case TIR are satisfied, so $\widetilde{Y}=A$ by the condition (A5) of Case TIR and since $\widetilde{\pi}:\widetilde{X}\to A$ has connected fibres.
Suppose next that $\delta_{\widetilde{f}_r}=\delta_{\widetilde{g}}$.
If $\dim(\widetilde{Y})=\dim(A)$, then $\widetilde{Y}=A$ and hence $\delta_f=\delta_{\widetilde{g}}=\delta_{g_A}<\delta_f$, a contradiction.

We are only left with the case where $\delta_{\widetilde{f}_r}=\delta_{\widetilde{g}}>\delta_{g_A}$ and $\dim(\widetilde{Y})=\dim(A)+1$.
Denote by $\tau:\widetilde{Y}\to A$ the induced morphism.
Note that $\widetilde{Y}$ is $\Q$-factorial and klt.
Since $\Sing(\widetilde{Y})$ does not dominate $A$ and the general fibre of $\widetilde{\pi}:\widetilde{X} \to A$ is a rational surface, the general fibre of $\tau$ is a smooth rational curve.
By the adjunction, the restriction of $K_{\widetilde{Y}}$ on the general fibre of $\tau$ is anti-ample.
Hence $K_{\widetilde{Y}}$ is not pseudo-effective.
Then we may run MMP starting from $\widetilde{Y}$ which ends up with $A$.
Since $\dim(\widetilde{Y})=\dim(A)+1$ and $\tau$ has connected fibres, any appearance of divisorial contraction or flip will create some proper $(\widetilde{\mathcal{I}}|_A)^{-1}$-periodic closed subvariety in $A$ by \cite[Lemma 7.5]{CMZ20}.
So $\tau$ has to be a Fano contraction by \cite[Lemma 4.1]{Men20}.
Note that $D_{\widetilde{Y}}:=(\widetilde{\pi}_r\circ \varphi)(\widetilde{Z})$ is an $(\widetilde{\mathcal{I}}|_{\widetilde{Y}})^{-1}$-periodic divisor in $\widetilde{Y}$ since $\widetilde{\pi}|_{\widetilde{Z}}:\widetilde{Z}\to A$ is finite surjective.
By Proposition \ref{prop-ti-kappa>0}, $\kappa(\widetilde{Y}, D_{\widetilde{Y}})>0$.
Note that $\widetilde{\pi}_r^{-1}(D_{\widetilde{Y}})$ is $\widetilde{\mathcal{I}}_r^{-1}$-periodic.
So $\widetilde{\pi}_r^{-1}(D_{\widetilde{Y}})\subseteq \widetilde{D}_r$ and hence $\kappa(\widetilde{X}_r, \widetilde{D}_r)>0$, a contradiction to Claim \ref{cl-Dr}.
The claim is proved.

\begin{cl}\label{tir3-cl-=1}
$\deg (\widetilde{\pi}_r|_{\widetilde{C}_r}: \widetilde{C}_r\to \widetilde{Y}=A) =1$ for any $\widetilde{\mathcal{I}}_r^{-1}$-periodic irreducible closed subvariety $\widetilde{C}_r$ of codimension $2$ in $\widetilde{X}_r$.
\end{cl}

By Claim \ref{cl-rep}, it suffices for us to show that $\widetilde{C}_r$ is dominated by some
$\widetilde{\mathcal{I}}^{-1}$-periodic irreducible closed subvariety of codimension $2$ in $\widetilde{X}$.
Note that $\widetilde{C}_r\subseteq \widetilde{D}_r=\varphi(\widetilde{D})$ by Claim \ref{cl-Dr}.
If $\dim(\varphi^{-1}(\widetilde{C}_r))=\dim (A)$, then we are done.
If $\dim(\varphi^{-1}(\widetilde{C}_r))=\dim (A)+1$, then $\varphi(E)=\widetilde{C}_r$ for some ($\varphi$-exceptional) irreducible component $E$ of $\widetilde{D}$.
Recall that $E$ is an irreducible component of $\widetilde{D}=\mu_X^{-1}(D)$.
So $\mu_X|_E:E\to D$ is finite surjective and there is at least one $\widetilde{\mathcal{I}}^{-1}$-periodic irreducible closed subvariety $\widetilde{C}\subseteq (\mu_X|_E)^{-1}(Z)$ of codimension $2$ in $\widetilde{X}$ dominating $A$.
Note that  $\varphi(E)=\varphi(\widetilde{C})=\widetilde{C}_r$.
So the claim is proved.

\par \vskip 0.5pc
{\it Completion of the proof of Theorem \ref{thm-tir-nn2}.}
By Claim \ref{tir3-cl-xrtir}, we may apply Claims \ref{tir3-cl-zti} and \ref{tir3-cl->=3} to $\widetilde{X}_r$. Then there are at least three $\widetilde{\mathcal{I}}_r^{-1}$-periodic irreducible closed subvarieties of codimension $2$ in $\widetilde{X}_r$ by Claim \ref{tir3-cl-=1}.
So we get a contradiction by applying Claim \ref{tir3-cl-<=2} to $\widetilde{X}_r$ again.
This completes the proof of Theorem \ref{thm-tir-nn2}.
\end{proof}

\begin{cor}\label{cor-iamp}
Let $X$ be a normal projective variety over $K$ such that $X_{ \overline{K}}$ has only $\Q$-factorial Kawamata log terminal (klt) singularities and admits an int-amplified endomorphism.
Then Conjecture \ref{conj_zf} holds for any surjective endomorphism on $X$ in the following cases:

\begin{itemize}
\item[(1)] The group $\pi_1^{\alg}((X_{ \overline{K}})_{\reg})$ is finite (e.g.~when $X$ is smooth and rationally connected).
\item[(2)] $X_{ \overline{K}}$ has only terminal singularities and $\dim(X)=3$.
\end{itemize}
\end{cor}

\begin{proof}
It follows from Theorems \ref{thm_int-amp-TIR}, \ref{thm_iamp-src} and \ref{thm-tir-nn2}.
Note that for (2), if $X_{ \overline{K}}$ has terminal singularities, then the birational MMP involves only with the terminal singularities (cf.~\cite[Corollary 3.43]{KM98}).
\end{proof}

In the rest of this section, we show the $\Q$-factorial property
of $\widetilde{X}$ (Proposition \ref{prop-nqf}) used in the proof of Theorem \ref{thm-tir-nn2}
(i.e., in \ref{n:6.1}).
A local ring $(R, \mathfrak{m})$ is {\it $\Q$-factorial} if for any prime ideal $\mathfrak{p}$ of height one, $\mathfrak{p} = \sqrt{(f)}$ for some $f \in p$.

\begin{lem}\label{lem-nqf-ti}
Let $\pi:X\to Y$ be a finite surjective morphism of normal varieties. Let $x\in X$ be a closed point and $y=\pi(x)$.
Suppose the (completion) local ring $\hat{\mathcal{O}}_{X,x}$ is $\Q$-factorial.
Then so is $\hat{\mathcal{O}}_{Y,y}$.
\end{lem}

\begin{proof}
Set $A:=\hat{\mathcal{O}}_{Y,y}$ and $B:=\hat{\mathcal{O}}_{X,x}$.
Let $\varphi:A\to B$ be the induced local homomorphism, which is an injective finite morphism (cf.~\cite[Proposition 10.13]{AM69}).
We may regard $A$ as a subring of $B$.
Let $\mathfrak{p}$ be any prime ideal of height $1$ on $A$. Take a prime ideal $\mathfrak{q}$ in $B$ of height $1$ such that $\mathfrak{p} = \mathfrak{q}\cap A$.
Since $B$ is $\Q$-factorial, we can take $\beta \in \frak{q}$ such that $\mathfrak{q} =\sqrt{\beta B}$.
Take any $\alpha \in \mathfrak{p}$.
Now $\alpha\in \mathfrak{q}=\sqrt{\beta B}$.
So $\alpha^k=c\beta$ for some $k>0$ and $c\in B$.
Taking norm, we have
$\alpha^{kr} = N(c)N(\beta)$ where $r = \dim_{Q(A)} Q(B)$.
Note that $N(B)\subseteq A$ since $\varphi$ is finite.
Hence we have $\mathfrak{p} \subseteq \sqrt{(N(\beta))A}$.
Conversely,
$N(\beta)\in \mathfrak{p}$
since $N(\beta) \in \mathfrak{q}\cap A=\mathfrak{p}$.
So we have $\mathfrak{p}= \sqrt{N(\beta)A}$.
\end{proof}

\begin{lem}\label{lem-nqf-all}
Let $X$ be a normal variety. Then we have the following:
\begin{itemize}
\item[(1)] Let $x\in X$ be a closed point. If $\hat{\mathcal{O}}_{X,x}$ is $\Q$-factorial, then so is $\mathcal{O}_{X,x}$.
\item[(2)] If $\mathcal{O}_{X,x}$ is $\Q$-factorial for any closed point $x\in X$, then $X$ is $\Q$-factorial.
\end{itemize}
\end{lem}

\begin{proof}
(1) Let $S=\Spec \mathcal{O}_{X,x}$ and $\hat{S}=\Spec \hat{\mathcal{O}}_{X,x}$.
Let $\pi:\hat{S}\to S$ be the completion morphism which is faithfully flat.
Since $S$ is normal, so is $\hat{S}$.
Take a prime divisor $D$ on $S$.
Since $\pi$ is flat, $\pi^*\mathcal{O}_S(D)$ is a reflexive sheaf of rank $1$ on $\hat{S}$.
So $\pi^*\mathcal{O}_S(D)=\mathcal{O}_{\hat{S}}(\hat{D})$ for some Weil divisor $\hat{D}$ on $\hat{S}$.
If $\hat{\mathcal{O}}_{X,x}$ is $\Q$-factorial, then $n\hat{D}$ is Cartier for some $n>0$.
Then $\pi^*\mathcal{O}_S(nD)\cong \pi^*((\mathcal{O}_S(D)^{\otimes n})^{\vee\vee})\cong (\pi^*\mathcal{O}_S(D)^{\otimes n})^{\vee\vee}\cong \mathcal{O}_{\hat{S}}(n\hat{D})$ is invertible.
By the fpqc descent along $\pi$, $\mathcal{O}_S(nD)$ is invertible.

(2) Take any prime divisor $D$ of $X$.
Take any closed point $x\in D$.
Take an affine neighbourhood $U=\Spec A\subseteq X$ of $x$.
Let $\mathfrak{p}\subseteq A$ be the prime ideal of height $1$ representing $D$ and $\mathfrak{m}\subseteq A$ the maximal ideal representing $x$.
By assumption, $\mathfrak{p}A_{\mathfrak{m}}=\sqrt{\alpha A_{\mathfrak{m}}}$ for some $\alpha\in \mathfrak{p}$.
Write $\mathfrak{p}=(f_1,\cdots, f_r)$.
Then there exsits $n>0$ such that $f_i^n\in \alpha A_{\mathfrak{m}}$ for all $i$.
Write $f_i^n=\alpha a_i/b_i$ for some $a_i\in A$ and $b_i\in A\backslash \mathfrak{m}$.
Then $f_i\in \sqrt{\alpha A_b}$ where $b=b_1\cdots b_r\in A\backslash \mathfrak{m}$.
In particular, $\mathfrak{p} A_b=\sqrt{\alpha A_b}$ and $\Spec A_b$ is an affine open neighbourhood of $x$.
\end{proof}

\begin{prop}\label{prop-nqf}
Let $\pi:X\to Y$ be a surjective morphism of normal projective varieties with connected fibres.
Let $f:X\to X$ and $g:Y\to Y$ be int-amplified endomorphisms such that $g\circ \pi=\pi\circ f$.
Suppose $Y$ is an abelian variety and $X$ has only terminal singularities.
Suppose further $\dim(X)\le \dim(Y)+2$.
Then $X$ is $\Q$-factorial.
\end{prop}

\begin{proof}
Denote by $S$ the set of closed points $x\in X$ such that $\hat{\mathcal{O}}_{X,x}$ is not $\Q$-factorial.
By Lemma \ref{lem-nqf-ti}, $f^{-1}(S)\subseteq S$.
Then $f^{-1}(\overline{S})\subseteq \overline{S}$ by \cite[Lemma 7.2]{CMZ20}.
Hence $f^{-1}(\overline{S})=\overline{S}$.
By \cite[Lemma 7.5]{CMZ20}, $\pi(\overline{S})$ is $g^{-1}$-periodic.
Note that $S\subseteq \overline{S}\subseteq \Sing(X)$ and $\dim(\Sing(X))\le \dim(X)-3<\dim(Y)$ since $X$ has terminal singularities.
So $\pi(\overline{S})=\emptyset$ and hence $S=\emptyset$ by \cite[Lemma 4.1]{Men20}.
By Lemma \ref{lem-nqf-all}, $X$ is $\Q$-factorial.
\end{proof}

\section{fibre spaces; UBC vs sAND; Proofs of Theorems \ref{ThmB} and \ref{ThmC}}
\label{sec_fib}

In this section, we study the set $Z_f$ for endomorphisms $f$ on a fibration which fixes fibres set-theoretically, i.e.~a family of endomorphisms.
First we formulate a relative version of the sAND Conjecture.

Let $X, S$ be quasi-projective varieties over $K$ and $\pi: X \to S$ a projective morphism.
Let $f:X\to X$ be a surjective morphism over $S$.
Take any immersion $X \longrightarrow P$ into a projective variety $P$, any ample divisor $H$ on $P$, and height $h_{H}\geq1$ associated with $H$.
For any point $x\in X( \overline{K})$,
setting $X_{\pi(x)} :=\pi^{-1}(\pi(x))$,
we define
$$
\alpha_f(x) := \lim\limits_{n \to +\infty}h_{H}(f^{n}(x))^{1/n}=\lim\limits_{n \to +\infty} h_{H|_{X_{\pi(x)}}}((f|_{X_{\pi(x)}})^{n}(x))^{1/n}= \alpha_{f|_{X_{\pi(x)}}}(x)$$
which is well-defined in this relative setting,
since the arithmetic degree is well-defined and independent of the choice of ample height on any projective scheme over a number field.
We define $Z_{f}$ and $Z_{f}(d)$ in the same way as in Definition \ref{defn3.1.2}.

\begin{Conjecture}[Relative sAND Conjecture]\label{conj_relsand}
Let $X, S$ be quasi-projective varieties over $K$ and $\pi: X \to S$ a projective morphism.
Let $f:X\to X$ be a surjective morphism
with $\pi \circ f=\pi$.
Then the set $Z_f(d)$ is not  Zariski dense in $X_{\overline{K}}$
for any constant $d>0$.
\end{Conjecture}

\begin{rem}
Suppose $\pi$ has geometrically irreducible fibres.
Then the dynamical degree of $f_{ \overline{K}} \colon X_{ \overline{K}} \longrightarrow X_{ \overline{K}}$ is equal to
the dynamical degree of $f|_{X_{s}} \colon X_{s} \longrightarrow X_{s}$ for general $s\in S( \overline{K})$,
by the product formula involving the relative dynamical degree
(cf.~\cite{DN11,Tru20}).
\end{rem}

\begin{rem}
If $X$ is projective, then Conjecture \ref{conj_relsand} is equivalent to Conjecture \ref{conj_zf}.
\end{rem}

First we consider projective bundles over projective bases.

\begin{thm}[cf.~\cite{Ame03}]\label{thm_projbdl}
Let $f$ be a surjective endomorphism on a projective bundle $\mathbb P_Y(E) \overset
{\pi}{\to} Y$ over a projective variety with $\pi \circ f=\pi$.
Then $Z_f(d)$ is not Zariski dense for any $d>0$.
\end{thm}

\begin{proof}
We may assume that $\delta_f>1$.
Let $X:=\mathbb P_Y(E)$ and assume $E$ has rank $r+1$.
By Lemma \ref{lem_rat}, we may always replace $Y$ (and hence $X$) by a generically finite cover since $f$ lifts.
So first we may assume $Y$ is smooth.
Note that $\pi^{-1}(y)\cong \mathbb{P}^r$ for any $y\in Y$.
Then $f|_{\pi^{-1}(y)}$ is $\delta_f$-polarized.
In particular, $\deg f|_{\pi^{-1}(y)}>1$.
By \cite[Theorem 1]{Ame03}, after an (\'etale) base change of $Y$, we may assume $X=Y\times \mathbb{P}^r$.
Then we have the well defined map
$$\Phi: Y \to R^m(\mathbb P^r, \mathbb P^r)\, \text{ via } \, y\mapsto f|_{\pi^{-1}(y)}$$
where $R^m(\mathbb P^r, \mathbb P^r)$ is the space of endomorphisms
given by degree $m$ polynomials on $\mathbb P^r$.
Note that $R^m(\mathbb P^r, \mathbb P^r)$ is an affine variety (cf.~\cite[Lemma 1.2]{Ame03}).
Then $\Phi$ is a constant map.
In particular, $f$ splits as $\id_Y\times \Phi(Y)$.
So the assertion follows from Lemma \ref{lem_prod} and Theorem \ref{thm_pol}.
\end{proof}

The next result gives the structure of the small arithmetic degree set $Z_f$
in the case of abelian fibrations, and is the key
in proving Theorem \ref{ThmC}.

\begin{thm}\label{thm_abfib}
Let $\pi: X \to Y$ be an abelian fibration, that is, a projective surjective morphism of quasi-projective normal varieties such that $\pi_*\mathcal O_X=\mathcal O_Y$ and its generic fibre $X_\eta$ is an abelian variety.
Let $f: X \to X$ be a surjective endomorphism with $\delta_f>1$ and $\pi \circ f=\pi$.
Then,
after enlarging the ground field by a finite extension and taking base change along a finite cover of $Y$,
there is a Zariski open dense subset $V \subset Y$ such that $X_{V}=\pi^{-1}(V) \longrightarrow V$ has an abelian scheme structure compatible with
the abelian variety structure of the generic fibre $X_{\eta}$, and there are
a non-trivial $V$-group homomorphism $\beta: X_V \to X_V$
and a section $M \subset X_{V}$ of $X_V \to V$ such that,  for any $d>0$, we have:
\begin{align*}
Z_f(d) \subseteq
\pi^{-1}(Y\setminus V) \cup \left( M(d)+ \beta^{-1}(\Tor(X_V(d))) \right) .
\end{align*}
Here $\Tor(X_V(d))$ is the set of rational points in $X(d)$ which are torsion points of the fibres over $V$ and
$ M(d)+ \beta^{-1}(\Tor(X_V(d))) =\{a+b \mid \pi(a)=\pi(b), a \in M(d), b\in \beta^{-1}(\Tor(X_V(d)))   \}$.
\end{thm}

Consider first the case when the endomorphism induces an isogeny on the generic fibre.

\begin{lem}\label{lem_abfib-isog}
Let $\pi: X \to Y$ be an abelian fibration and
$f: X \to X$ a surjective endomorphism such that $\delta_f>1$, $\pi \circ f=\pi$, and the induced endomorphism $f_\eta: X_\eta \to X_\eta$ on the generic fibre of $\pi$ is an isogeny.
Then, after enlarging the ground field by a finite extension and taking base change along a finite cover of $Y$,
there are a Zariski open dense subset $V \subset Y$ over which $X_{V}=\pi^{-1}(V) \longrightarrow V$ has an abelian scheme structure restricting to
the abelian variety structure of the generic fibre $X_{\eta}$
 and a non-trivial $V$-group homomorphism $\beta: X_V \to X_V$
 such that, for any $d>0$, we have:
\begin{align*}
Z_f(d) \subseteq \pi^{-1}(Y \setminus V) \cup \beta^{-1}(\Tor(X_V(d))) .
\end{align*}
\end{lem}

\begin{proof}
Let $X_{\overline{\eta}}$ be the geometric generic fibre, and $f_{\overline{\eta}}$ the induced endomorphism.
Then, by a calculation of intersection numbers, we have $ \delta_{f_{\overline{\eta}}}= \delta_{f}= \delta_{f_{y}}$ where $y$
is a general closed point of $Y$. Denote this number as $ \delta$.

We can find
nef $\mathbb R$-divisors $ \overline{D} \not\equiv 0$ on $X_{\overline{\eta}}$ such that
$f_{\overline{\eta}}^{*} \overline{D} \sim_{\R} \delta \overline{D}$ (cf.~\cite[Remark 5.11]{San17}).
Fix an ample divisor $H$ on $X$.
By Lemma \ref{lem_decomp}, there exists some $ \overline{ \alpha} \in \End(X_{\overline{\eta}})_{\R}$
such that $\Phi^{H_{ \overline{\eta}}}_{ \overline{D}}=i_{H_{ \overline{\eta}}}( \overline{ \alpha})\circ \overline{\alpha}$
and $i_{H_{ \overline{\eta}}}( \overline{ \alpha})= \overline{\alpha}$.
Note that $ \overline{\alpha}$ is not zero since $ \overline{D} \not\equiv 0$.

Take a sufficiently large finite extension $F$ of $k(\eta)$ so that
$ \overline{D}$ and $ \overline{ \alpha}$ descend to an $\R$-Cartier divisor $D_{F}$ on $X_{F}$ and $ \alpha_{F}\in \End(X_{F})_{\R}$
respectively.
Moreover, we can take $F$ so that $f_{F}^{*}D_{F}\sim_{\R} \delta D_{F}$,
 $\Phi^{H_{F}}_{D_{F}}=i_{H_{F}}( \alpha_{F})\circ\alpha_{F}$, and $i_{H_{F}}( \alpha_{F})=\alpha_{F}$.
(For the definition of  $\Phi^{H_{F}}_{D_{F}}$, see Lemma \ref{lem_spreadout} and Remark \ref{rem_phi}).
Represent $ \alpha_{F}$ as follows:
 \[
 \alpha_{F}=\sum_{i=1}^{r}\beta_{i,F}{\otimes}c_{i}
 \]
 where $\beta_{i,F}\in \End(X_{F})$ and $c_{i}\in \R$ are linearly independent over $\Q$.

Replacing $Y$ with its normalization in $F$ and $X$ with the base change, we may assume $F=k(\eta)$.
(Here we also extend the ground field so that $Y$ is geometrically irreducible.)
Moreover, by shrinking $Y$, we may assume $ \pi \colon X \longrightarrow Y$ has an abelian scheme structure compatible
with the one on the generic fibre. Further, one can find:
\begin{itemize}
\item an $\R$-Cartier divisor $D$ on X which restricts to $D_{F}$; and
\item $\beta_{i} \in \End(X/Y)$, group endomorphisms of abelian group scheme $X$ over $Y$ restricting to $\beta_{i,F}$
\end{itemize}
such that $f^{*}D \sim_{\R} \delta D$,
 $\varphi_{H}^{-1}\circ \varphi_{D}=i_{H}( \alpha)\circ \alpha$ in $\End(X/Y)_{\R}$
 where $ \alpha=\sum_{i=1}^{r}\beta_{i} {\otimes} c_{i}$,
 and $i_{H}( \alpha)= \alpha$.

Take any $y\in Y(\overline K)$.
On $X_{y}$ we have
\begin{itemize}
\item $f_{y}^{*}D_{y} \sim_{\R} \delta D_{y}$;
\item $\varphi_{H_{y}}^{-1}\circ \varphi_{D_{y}}=i_{H_{y}}( \alpha_{y})\circ \alpha_{y}$, $i_{H_{y}}( \alpha_{y})= \alpha_{y}$
in $\End(X_{y})_{\R}$; and
\item $ \alpha_{y}=\sum_{i=1}^{r} \beta_{i,y} {\otimes} c_{i}$.
\end{itemize}

Now, since the isogeny $f_y$ commutes with the multiplication map $[m]$ (with $m \ge 2$) on ${X_y}$, by \cite[Exercise B.5]{HS00},
we have
\[
\hat{h}_{D_{y}}(x)=\lim_{n\to +\infty }\frac{h_{D_{y}}(f_{y}^{n}(x))}{ \delta^{n}}=\lim_{n' = m^n\to +\infty}\frac{h_{D_{y}}(n' x)}{(n')^{2}}
\]
for all $x\in X_{y}( \overline{K})$.
Then as in the proof of \cite[Theorem 29]{KS16b}, we have
\[
\hat{h}_{D_{y}}(x)=\frac{1}{2} \left< \alpha_{y}(x), \alpha_{y}(x) \right>_{H_{y}}
\]
where $\left<, \right>_{H_y}$ is the height pairing associated with the quadratic part of the canonical height associated with $H_{y}$.

Fix some $d>0$.
Since $\left<, \right>_{H_y}$ is positive definite on $X_{y}(d) {\otimes}_{\Z}\R$,
for $x \in X_{y}(d)$, we have

\begin{align*}
\alpha_{f}(x)< \delta\  &\Longrightarrow\  \alpha_{f_{y}}(x)< \delta\  \Longrightarrow
\hat{h}_{D_{y}}(x)=0 \\
 &\Longrightarrow\  \alpha_{y}(x)=0 \text{ in $X_{y}(d) {\otimes}_{\Z}\R$} \\
&\Longrightarrow\ \sum_{i=1}^{r} \beta_{i,y}(x) {\otimes} c_{i}=0 \text{ in $X_{y}(d) {\otimes}_{\Z}\R$}.
\end{align*}

Since $c_{i}$ are linearly independent over $\Q$, the last equality implies
\begin{align*}
\beta_{i,y}(x)=0 \text{ in $X_{y}(d) \otimes_{\Z} \Q$ for all $i$}.
\end{align*}

There is at least one $\beta := \beta_i$ which is a non-zero $Y$-group endomorphism of the abelian scheme $X \longrightarrow Y$.
Then we have
$Z_f(d) \subseteq \beta^{-1}(\Tor(X(d)))$.
\end{proof}

\begin{proof}[Proof of Theorem \ref{thm_abfib}]
First we apply Lemma \ref{lem_sil} and Remark \ref{split_nonisog} to the endomorphism $f_{\overline \eta}: X_{\overline \eta} \to X_{\overline \eta}$ on the geometric generic fibre of $\pi$.
Then, enlarging the ground field by a finite extension and $Y$ with a Zariski open dense subset of its finite cover,
we may assume:

\begin{itemize}
\item $X=X_{1}\times_{Y}  X_{2}$ where $\pi_{i} : X_{i} \to Y$ are abelian scheme over $Y$ ($i=1,2$);
\item $f=f_{1} \times_Y  \ f_{2}$ where $f_{i}$ are surjective endomorphism of $X_{i}$ over $Y$ ($i=1,2$);
\item  there is a section $p_1: Y \to X_{1}$ such that
$g_1=\tau_{p_1}^{-1}\circ f_{1} \circ \tau_{p_1} : X_{1} \to X_{1}$ is a $Y$-group endomorphism of the abelian scheme $\pi_1: X_{1} \to Y$; and
\item $\delta_{f_2}=1$.
\end{itemize}

By Lemma \ref{lem_abfib-isog}, after enlarging the ground field by a finite extension and replacing $Y$ with a Zariski open dense subset of its finite cover again,
there exists a non-trivial $Y$-group homomorphism $\beta_1: X_1 \to X_1$  such that
$$Z_{g_1}(d) \subseteq  \beta_1^{-1}(\Tor(X_1(d))).$$
Then we have $Z_{f_1}(d) \subseteq M_1(d)+\beta_1^{-1}(\Tor(X_1(d)))$,
where $M_1 \subset X_1$ is the image of $p_1$.

Take a general point $y \in Y(d)$.
Then $$Z_{f_y}(d) \subseteq \left( p_1(y)+\beta_1^{-1}(\Tor((X_1)_y(d))) \right) \times (X_2)_y(d).$$
Extend $p_1, \beta_1$ to $p:Y \to X$, $\beta: X \to X$ by
$p(y)=(p_1(y),0)$, $\beta(x)=(\beta_1(x),0)$.
Let $M \subset X$ be the image of $p$.
Then we have
$Z_f(d) \subseteq M(d)+\beta^{-1}(\Tor(X(d)))$.
\end{proof}

Now we are ready to prove Theorem \ref{ThmC}.

\begin{proof}[Proof of Theorem \ref{ThmC}]
By Theorem \ref{thm_abfib}, after enlarging the ground field by a finite extension and replacing $Y$ with
a Zariski open dense subset of its finite cover,
there are a section $M \subset X$ of $\pi$ and a non-trivial $Y$-group homomorphism
$\beta: X \to X$ such that
$Z_f(d) \subseteq M(d)+\beta^{-1}(\Tor(X(d)))$.
By assumption, the multiplication by a sufficiently divisible integer $N$ maps all points of $\Tor(X(d))$ to
the zero section $E$.
So we have
$Z_f(d) \subseteq M(d) +  (\beta^{-1}([N]^{-1}E))$.
Thus $Z_f(d)$ is not Zariski dense in $X_{\overline{K}}$.
\end{proof}

As corollaries, Conjecture \ref{conj_relsand} for elliptic fibrations and trivial abelian fibrations follow.

\begin{thm}\label{thm_ellfib}
Let $f$ be a surjective endomorphism on an elliptic fibration $X \overset
{\pi}{\to} Y$ over a quasi-projective variety with $\pi \circ f=\pi$.
Then $Z_f(d)$ is not Zariski dense in $X_{\overline{K}}$ for any $d>0$.
\end{thm}

\begin{thm}\label{thm_trivabfib}
Let $f$ be a surjective endomorphism on a trivial abelian fibration $X=A \times Y \overset
{\pi}{\to} Y$ over a quasi-projective variety with $\pi \circ f=\pi$.
Then $Z_f(d)$ is not Zariski dense in $X_{\overline{K}}$ for any $d>0$.
\end{thm}

Next, we consider endomorphisms on projective varieties of non-negative Kodaira dimension.
The monoid $\SEnd(X)$ of surjective endomorphisms of a variety of general type is a finite set, so Conjecture \ref{conj_zf} holds for them.
For a smooth projective variety $X$ with $\kappa(X)=\dim (X)-1$,
we can prove Conjecture \ref{conj_zf} as an application of Theorem \ref{thm_ellfib}.

\begin{thm}\label{thm_kappa=dim-1}
Let $X$ be a smooth projective variety with $\kappa(X)=\dim (X)-1$ and $f: X \to X$ a surjective endomorphism.
Then $Z_f(d)$ is not Zariski dense in $X_{\overline{K}}$ for any $d>0$.
\end{thm}

\begin{proof}
By \cite[Theorem A]{NZ09}, there is an automorphism $g$ of finite order on the base $W$ of an
Iitaka fibration $\phi: X \dashrightarrow W$,
such that $\phi \circ f=g \circ \phi$.
Replacing $f$ by a positive power, we may assume that $g = \id$ and hence $\phi \circ f=\phi$.
By Lemma \ref{lem_rat}, we may replace $X$ with the normalization $\Gamma$ of the graph of $\phi$ and $f$ with its lifting to $\Gamma$. Thus we may assume that $\phi$ is an elliptic fibration.
Then apply Theorem \ref{thm_ellfib}.
\end{proof}

For a non-invertible endomorphism on a threefold $X$ with $\kappa(X) \geq 0$,
we can show the following weak version of non-density
thanks to the classification results by Fujimoto and Nakayama (\cite{Fuj02}, \cite{FN07}).

\begin{thm}\label{thm_kappa2}
Let $X$ be a smooth projective threefold with Kodaira dimension $\kappa(X) \geq 0$ and $f: X \to X$ a non-invertible surjective endomorphism.
Then
$Z_f(L)=\{ x \in X(L) \mid \alpha_f(x) < \delta_f \}$
is not Zariski dense in $X_{\overline{K}}$ for any finite extension field ($\overline{K} \supseteq$) $L \supseteq K$.
\end{thm}

\begin{proof}
The proof is in steps.

\par \vskip 1pc \noindent
{\bf Step 1.}
By \cite{Fuj02}, we can run an $f$-equivariant MMP
$$X=X_0 \to X_1 \to \cdots \to X_r=Y$$
such that each step is a blowdown of
a smooth surface to an elliptic curve
(so each $X_i$ is smooth) and $K_Y$ is nef and hence semi-ample by the known abundance for threefolds.
By Lemma \ref{lem_surj} and
replacing $X$ by $Y$, we may assume that $K_X$ is semi-ample.
Take the Iitaka fibration $\phi: X \to C$.
By \cite[Theorem A]{NZ09}, replacing $f$ by a positive power, we may assume that $\phi \circ f=\phi$.
Let $F$ be a general fibre of $\phi$. Then $K_F \sim_{\Q} 0$.

If $\kappa(X)=3$, then $f$ is an automorphism of finite order, contradicting the assumption.
The case $\kappa(X)=2$ is a special case of Theorem \ref{thm_kappa=dim-1}.

Assume $\kappa(X)=0$.
We use the Beauville--Bogomolov decomposition theorem over $\overline K$ (\cite[Proposition 3.1]{BGRS17}).
Then there is a finite \'etale cover $\phi: \tilde X \to X$ where $\tilde X$ is either an abelian threefold or the product of an elliptic curve and a K3 surface,
and $f$ lifts to an endomorphism $\tilde f$ on $\tilde X$
(cf.~\cite[Section 4]{LS21}).
Moreover, if $\tilde X$ is the product of an elliptic curve and a K3 surface,
then $\tilde f$ splits.
So the assertion follows from Lemma \ref{lem_surj}, Lemma \ref{lem_prod}, and Theorem \ref{thm_ab2}.

From now on, we assume $\kappa(X)=1$.
A general fibre $F$ has a non-invertible (necessarily \'etale) endomorphism $f|_F$, so it is either an abelian or a hyperelliptic surface, by surface theory or \cite[Theorem 3.2]{Fuj02}.

\par \vskip 1pc \noindent
{\bf Step 2.}
Assume that $F$ is a hyperelliptic surface.
By Fujimoto \cite[Theorem 5.10]{Fuj02}, $\phi$ decomposes as
$\phi: X \overset{\pi}{\to} T \overset{q}{\to} C$ where $T$ is a normal projective surface,
$\pi$ is an elliptic fibration, and $q$ is a $\mathbb P^1$-fibration.
Moreover, $f$ descends to a morphism $g: T \to T$.
Take a Zariski open dense subset $U \subset C$ such that
$\phi_U=\phi|_{X_U}:X_U =\phi^{-1}(U) \to U$ is smooth.
Then we can take the relative Albanese morphism $a: X_U \to A_U$ over $U$
(cf.~\cite[Proposition 9.5.20]{FAG05}).
Let $h_U: A_U \to A_U$ be the morphism induced from $f$.
Setting $T_U=q^{-1}(U)$, the natural morphism $\theta: X_U \to T_U \times_U A_U$ is induced.
Taking the fibres at a general point $c \in U$,
we can see that
$\theta_c: X_c \to T_c \times A_c$ is a finite  surjective morphism because of the hyperelliptic surface structure of $X_c$.
So $\theta$ is a finite surjective morphism.

Consider the endomorphism $g_U \times_U h_U$ on $T_U \times_U A_U$.
Since $f|_C = \id$ and by the product formula (cf.~\cite{DN11} and \cite{Tru20}), the dynamical degree of $g$ (resp.~$h_U, g_U \times_U h_U$) is equal to
the dynamical degree of the restriction $g_c$ (resp.~$h_c, g_c \times h_c$)
of  $g$ (resp.~$h_U, g_U \times_U h_U$) to the fibre at a general point $c \in U$.
Further, taking a general point $c \in U$, we have
$$\delta_{g_U \times_U h_U}=\delta_{g_c \times h_c} =\max \{ \delta_{g_c}, \delta_{h_c} \}
=\max \{\delta_g, \delta_{h_U} \}.$$
Now $\delta_f=\delta_{g_U \times_U h_U}$ since $\theta$ is finite surjective.
Therefore $\delta_f=\max \{\delta_g,\delta_{h_U}\}$.

Note that $Z_g(d)$ and $Z_{h_U}(d)$ are not Zariski dense
(Theorem \ref{thm_surf} and Theorem \ref{thm_ellfib}).
If $\delta_f=\delta_g$, then $Z_f(d) \subset \pi^{-1}(Z_g(d))$ is not Zariski dense.
Similarly if $\delta_f=\delta_{h_U}$, then
by applying Lemma \ref{lem_surj} fibrewise, we get $Z_f(d) \subseteq \phi^{-1}(C \setminus U) \cup a^{-1}(Z_{h_U}(d))$ and again
$Z_f(d)$ is not Zariski dense.

\par \vskip 1pc \noindent
{\bf Step 3.}
Assume that $F$ is an abelian surface.
If $F$ is a simple abelian surface or $\phi$ is primitive,
then $\phi$ is a Seifert abelian fibration in the sense of \cite[Definition 2.3]{FN07} (cf.~\cite[Theorem 4.1]{FN07}, \cite[Theorem 4.3 (1)]{FN07}).
Then $X$ admits an finite \'etale cover $W \to X$ from a smooth abelian fibration $W \to T$.
Now $\kappa(T)=\kappa(W)=\kappa(X)=1$ (cf.~e.g.~\cite[Lemma 2.1]{FN07}).
Thus $T$ is a curve of genus $\geq 2$, so it is not potentially dense by Faltings theorem.
So $W$ is not potentially dense, too.
The Chevalley--Weil theorem implies that $X$ is not potentially dense.
So $Z_f(L)$ is not Zariski dense for any finite extension field $L \supseteq K$.

Finally assume that $F$ is an abelian surface and $\phi$ is imprimitive.
By \cite[Theorem 5.6 and Corollary 5.8]{FN07},
$\phi$ decomposes as $X \overset{\pi}{\to} T \overset{q}{\to} C$ where
$T$ is a normal projective surface and $\pi, q$ are elliptic fibrations.
\cite[Proposition 5.7]{FN07} implies the following commutative diagram:

\[
\xymatrix{
\tilde X \ar[r]^{\alpha} \ar[d]_{\tilde \pi} & X \ar[d]^\pi  \\
\tilde T \ar[r]^\beta \ar[d]_{\tilde q} & T \ar[d]^q \\
\tilde C \ar[r]^\gamma & C
}
\]
where $\gamma$ is a finite Galois cover, $\tilde T, \tilde X$ are the normalizations of $T \times_C \tilde C, X \times_C \tilde C$, respectively, and $\tilde T$ is isomorphic to $E \times \tilde C$ for some elliptic curve $E$.
Let
$\rho: \tilde X \to E$ be the composition of $\tilde \pi: \tilde X \to \tilde T$ with the projection $\tilde T \cong E \times \tilde C \to E$.
By \cite[Theorem 5.10]{FN07},
there is a finite \'etale cover $\nu: E' \to E$ such that $\tilde X'=\tilde X \times_E E'$ is isomorphic to $E' \times S$ for some projective surface $S$ and $f$ on $X$ lifts to endomorphism $\tilde f': E' \times S \to E' \times S$ which splits as $\tilde f'=\phi' \times v$
where $\phi'$ is an automorphism on $S$ and $v$ is an endomorphism on $E$.
Then $Z_{\tilde f'}(d)$ is not Zariski dense by
Theorem \ref{thm_pol} and Theorem \ref{thm_surfauto}.
Now $\tilde X'$ is an equivariant \'etale cover of $X$,
so the theorem follows from Lemma \ref{lem_surj}.
\end{proof}

Finally, we show Theorem \ref{ThmB},
the equivalence of the uniform boundedness conjecture and the relative sAND conjecture for a family of polarized endomorphisms.

\begin{proof}[Proof of Theorem \ref{ThmB}]
\textbf{``(2) $\Rightarrow$ (1)''}
We can take a quasi-projective variety $U$ with a morphism $\phi: \mathbb P^N_U \to \mathbb P^N_U$ over $U$ which parametrizes all the endomorphisms on $\mathbb P^N$ of degree $r$.
Then any morphism $\varphi \colon \PP^{N} \longrightarrow \PP^{N}$ of degree $r$ defined over $L$, which is 
an extension of $K$ of degree $d$, corresponds to a point $s \in U(d)$, i.e. $\phi_{s} = \varphi$.
Note that $\phi_{s}$-preperiodic $L$-rational points are exactly the $L$-rational points at which 
$\phi_{s}$ is of arithmetic degree $1$ (or equivalently of arithmetic degree strictly less than $r$ since $\phi_{s}$ is polarized).
Hence we are done.

\textbf{``(1) $\Rightarrow$ (3)''}
Let $\eta \in S$ be the generic point.
Now $f_\eta: X_\eta \to X_\eta$ is a polarized endomorphism.
So $X_\eta$ can be embedded into a projective space over $f_\eta$-equivariantly
(cf.~\cite[Proposition 2.1]{Fak03}).
After shrinking $S$, we may assume that $X$ is a closed subvariety of $\mathbb P^N_S$ and $f$ extends to a surjective endomorphism $\phi: \mathbb P^N_S \to \mathbb P^N_S$ over $S$.

We claim that $Z_\phi(d)$ is contained in the closed subset
$$W=\{ x \in \mathbb P^N_S \mid \phi^m(x)=\phi^n(x) \}$$
 for some $m > n \geq 0$.
 Indeed, write the projection $\pi \colon \PP^{N}_S \longrightarrow S$ and take any point $x \in Z_{\phi}(d)$.
 Then $\pi(x) \in S(d)$ and $\phi_{\pi(x)} \colon \PP^{N} \longrightarrow \PP^{N}$ is defined over
 a field with extension degree at most $d$ over $K$.
 Note that $x \in Z_{\phi_{\pi(x)}}(d)$ as well and therefore $x$ is $\phi_{\pi(x)}$-preperiodic.
 Hence by assumption (1), the cardinality of the $\phi_{\pi(x)}$-orbit of $x$ is bounded depending only on $d, N$, and the degree of $\phi$.
 This implies that there are only finitely many possibilities of the $\phi$-orbit structure of points in $Z_{\phi}(d)$. 
 Thus it is enough to take $n$ larger than the maximum length of the orbits and take $m$ so that
 $m-n$ is divisible by all the orbits lengths.
  
Since $f=\phi|_X$ is of infinite order, $X \not\subseteq W$.
Hence $Z_f(d)=Z_\phi(d) \cap X \subseteq W \cap X$, a proper closed subset of $X$.
So $Z_f(d)$ is not Zariski dense in $X_{\overline K}$.

\textbf{``(3) $\Rightarrow$ (2)''}
We argue by induction on $\dim (X)$.
By assumption, $W=\overline{Z_f(d)}$ is a proper $f$-invariant closed subset of $X$.
Take the irreducible decomposition $W=W_1 \cup \cdots \cup W_p \cup W'_1 \cup \cdots \cup W'_q$ such that $W_1, \ldots, W_p$ are $f$-periodic and $W'_j$ maps to $\bigcup_{i=1}^p W_i$ by iteration of $f$ for $1 \leq j \leq q$.
Then $Z_f(d) \subseteq f^{-k}(\bigcup_{i=1}^p W_i)$ for some $k \geq 0$.

Take a sufficiently large $m>0$ such that $W_i$ is $f^m$-invariant for $1 \leq i \leq p$.
Set $g_i=f^m|_{W_i}$.
We may assume that
$\dim (W_i) = \dim (\pi(W_i))$ for $1 \leq i \leq r$ and
$\dim (W_i) > \dim (\pi(W_i))$ for $r+1 \leq i \leq p$.
Fix $i $ with $1 \leq i \leq r$.
Then $g_i$ is an automorphism of finite order on $W_i$ since $g_i$ induces an endomorphism on a general fibre of $\pi|_{W_i}$ which is a finite set.
Take some $M>0$ such that $g_i^M=\id_{W_i}$ for each $i$.
If there is a curve $C \subseteq W_i$ contracted by $\pi$,
then $0< (C \cdot H) = (C \cdot (g_i^{M})^{*}H)=r^{mM}(C \cdot H)$, which is a contradiction.
So $\pi|_{W_i}$ is finite for $1 \leq i \leq r$.

Now we have
$$Z_f(d)\subseteq f^{-k} \left( \bigcup_{i=1}^r W_i(d) \cup \bigcup_{j=r+1}^p Z_{g_j}(d) \right).$$
By induction hypothesis,
the number of the points in $Z_{g_j}(d) \cap X_s$ is uniformly bounded
for $r+1 \leq j \leq p$.
Moreover the number of the points of $W_i(d) \cap X_s$ is also uniformly bounded for
$1 \leq i \leq r$.
So the assertion (2) follows.
\end{proof}

\section{Further Generalizations}\label{Sect_ext}

We can generalize Conjecture \ref{conj_zf} to a more precise one:

\begin{Conjecture}\label{Conjecture2}
Let $f:X \to X$ be a surjective endomorphism on a projective variety
with the first dynamical degree $\delta_f>1$.
Let $\{ Z_\lambda \}_{\lambda \in \Lambda}$ be the collection of all $f$-preperiodic proper subvarieties of small dynamical degree in $X_{\overline{K}}$
(cf.~\S \ref{n:2.1} for the definition of subvarieties of small dynamical degree).
Then, for any $d>0$,
$Z_f(d) \subseteq Z_{\lambda_1} \cup \cdots \cup Z_{\lambda_m}$
for some $\lambda_1, \ldots, \lambda_m \in \Lambda$.

\end{Conjecture}

\begin{rem}
Conjecture \ref{Conjecture2} holds for the following cases.
\begin{itemize}
\item[(1)]
Polarized endomorphisms  (cf.~Theorem \ref{thm_pol}).
\item[(2)]
Surjective endomorphisms on abelian varieties (cf.~Theorem \ref{thm_ab2}).
\item[(3)]
Automorphisms on projective surfaces (cf.~Theorem \ref{thm_surfauto}).
\end{itemize}
In fact, in these cases, the respective theorems have given precise descriptions of $Z_f$.
\end{rem}

Clearly Conjecture \ref{Conjecture2} implies Conjecture \ref{conj_zf}.
For the converse, we have:

\begin{prop}
Fix a positive integer $n$.
If Conjecture \ref{conj_zf} holds in dimension $\leq n$,
then Conjecture \ref{Conjecture2} holds in dimension $n$.
\end{prop}

\begin{proof}
Let $f: X \to X$ be a surjective endomorphism on a projective variety $X$ of dimension $n$ with $\delta_f >1$.
Take any $d>0$.
Then $Z_f(d)$ is not Zariski dense by assumption.
Let $W$ be the Zariski closure of $Z_f(d)$ in $X_{\overline{K}}$.
We can take the irreducible decomposition of $W$ as $W=\bigcup_i W_i \cup \bigcup_j W_j'$ where each $W_i$ is $f$-periodic and each $W_j'$ is mapped to $\bigcup_i W_i$ by iteration of $f$.
Then we have $Z_f(d) \subseteq f^{-s}(\bigcup_i W_i)$ for some $s>0$.

Replacing $f$ by a positive power, we may assume that $W_i$ is $f$-invariant for each $i$.
Set $g=f|_{W_i}$.
By assumption, $Z_f(d) \cap W_i$ is Zariski dense in $W_i$.
So Conjecture \ref{conj_zf} for $g$ implies $\delta_g < \delta_f$.
Now we have $Z_f(d) \subseteq f^{-s}(\bigcup_i W_i)$, and each irreducible component of $f^{-s}(\bigcup_i W_i)$ is an $f$-preperiodic subvariety of small dynamical degree.
So we are done.
\end{proof}

Conjecture \ref{Conjecture2} entails that studying $f$-preperiodic subvarieties (especially those of small dynamical degree) is important in order to understand $Z_f$.
We conclude this section with some questions on $f$-preperiodic subvarieties.

\begin{ques}\label{ques_subvar}
Let $f: X \to X$ be a surjective endomorphism on a projective variety with $\delta_f>1$.
Fix  $d>0$.
\begin{itemize}
\item[(1)]
Let $Y$ be an $f$-invariant subvariety of $X_{\overline{K}}$.
Then, is $X(d) \cap (f^{-s}(Y) \setminus f^{-(s-1)}(Y))$  empty for $s \gg 0$?
\item[(2)]
Is there a positive integer $N$ such that $Y(d)$ is empty for any $f$-periodic subvariety $Y \subseteq X_{\overline K}$ of small dynamical degree whose period is larger than $N$?
\item[(3)]
Is the number of maximal $f$-invariant subvarieties of small dynamical degree in $X_{\overline{K}}$ finite?
\end{itemize}
\end{ques}

\begin{rem}
(1) and (2) are arithmetic questions but (3) is a purely geometric question.
Note that (3) is true for abelian varieties (cf.~Theorem \ref{thm_ab2}) and projective surfaces (cf.~Theorem \ref{thm_surfauto} and \cite[Section 5]{MZ19a}).
We refer the readers to \cite{MMSZZ20} for a recent progress of (3).
\end{rem}

\end{document}